\documentclass[a4paper,11pt]{article}
\usepackage{amssymb,amsmath,latexsym,amsthm}

\usepackage{xcolor}
\usepackage{url}
\usepackage{enumerate}
\usepackage[square,sort,comma,numbers]{natbib}
\usepackage{dsfont, esint}
\usepackage{diagbox}
 \usepackage{geometry}
\usepackage{graphicx}
\usepackage{colortbl}
\usepackage{bbm}

\usepackage{endnotes}

\renewcommand{\textbf}[1]{\begingroup\bfseries\mathversion{bold}#1\endgroup}

\usepackage{breqn}

\usepackage[colorlinks=true]{hyperref}
\usepackage[nameinlink]{cleveref}
\usepackage{fancyhdr}
\usepackage{pstricks,pst-plot,pst-node,pstricks-add}

\newtheorem{thma}{Theorem}

\newtheorem{thm}{Theorem}[section]
\newtheorem{defi}{Definition}[section]

\newtheorem{corollary}[thm]{Corollary}
\newtheorem{prop}[thm]{Proposition}

\newtheorem{lemma}[thm]{Lemma}
\newtheorem{claim}[thm]{Claim}

\theoremstyle{definition}
\newtheorem{remark}[thm]{Remark}

\newcommand{\R}{\mathbb R}

\newcommand{\Z}{\mathbb Z}
\newcommand{\N}{\mathbb N}
\newcommand{\C}{\mathbb C}

\numberwithin{equation}{section}

\def\XXint#1#2#3{{\setbox0=\hbox{$#1{#2#3}{\int}$}
    \vcenter{\hbox{$#2#3$}}\kern-.5\wd0}}

\allowdisplaybreaks

\makeatletter
\def\blfootnote{\xdef\@thefnmark{}\@footnotetext}
\makeatother

\date{date}
\begin{document}

\title{Effective recovery of Fourier spectra and spectral approximation by finite groups}

\author{Mircea Petrache\footnote{PUC, \texttt{mpetrache@mat.uc.cl}. ORCID id: 0000-0003-2181-169X} and 
Rodolfo Viera\footnote{PUC, \texttt{rodolfo.viera@mat.uc.cl}. ORCID id: 0000-0002-0810-6374}}\blfootnote{Pontificia Universidad Católica de Chile, Facultad de Matematicas, Avda. Vicuña Mackenna 4860, Santiago, 6904441, Chile}

\date{\today}
\maketitle
\begin{abstract}
    We prove a result on approximate recovery, with high probability, of subgroups of a finite nonabelian group $\Gamma$ from their random perturbations. We use this for ad-hoc sequences of $\Gamma_n$ while passing to the continuum limit, in order to obtain asymptotic almost sure recovery for rational lcsc nilpotent Lie groups. By comparison to limit theorems for groups of polynomial growth, it turns out that this setting is the natural general setting for recovery results, under polynomial growth assumptions on the $\Gamma_n$. This approach makes effective the convergence rate in previous Fourier recovery theorems in Euclidean space, and extends them to the nonabelian setting. A series of interesting further directions are highlighted by this approach.
\end{abstract}

\section{Introduction}\label{sec:intro}
Motivated by problems in signal denoising \cite{signoise}, hyperuniform point configurations \cite{kkt} and diffraction theory \cite{Hof, BG1}, to mention a few, in the last years there has been a lot of activity in random perturbations on discrete structures. Roughly speaking, the general framework is that, starting with a discrete set, one is concerned with how a structure of interest defined on it is affected by random perturbations of the set; for instance, in \cite{kkt} the authors relate by a very precise formula the structure factors of a (randomized) lattice and a perturbed version of it. Very recently, \cite{Yakir, pv} addressed the problem of studying the effect of random perturbations on the Fourier transform of a discrete set $X$ in the Euclidean space, proving robustness-type results for the Fourier transform in lattices \cite{Yakir} and more general quasi-periodic subsets of Euclidean spaces, such as quasicrystals \cite{pv}. The above works leave open the following natural research questions, which we tackle here:
\begin{itemize}
    \item In defining Fourier transform on countable uniformly discrete subsets of Euclidean space, one makes an implicit idealization, since in (numerical and physics) applications we never directly work with infinite sets, but are bound to do only \emph{approximations} to Fourier transforms on infinite spaces. It is thus natural to develop effective approximation error bounds for recovery results of structures such as the Fourier spectrum.
    \item While lattices and quasicrystals in Euclidean spaces are structures relevant to crystallographic applications, it is natural to ask general structure recovery results of the same kind, in general geometries, as done in \cite{signoise} in a different framework. In particular, it is natural to generalize such questions to the framework of nonabelian groups of polynomial growth.
\end{itemize} 
In this work, we tackle both questions for the first time, and at the same time, we set a basis for applying geometric group theory tools to recovery questions. In this way, through robustness-type results, we intend to contribute to understanding how probabilistic aspects in a group are determined by its geometry; we refer the reader to \cite{probgroup} for an introduction to these topics. While there is a large and deep theory for the possible limits of discrete groups of polynomial growth \cite{scallim, pansu}, the question of how to perform approximations that allows to ensure also convergence of further structures (such as Fourier transforms) is to our knowledge not developed. We hope that this work will motivate further advances towards a Fourier theory of approximate groups, by providing a concrete use case for such theory in the setting of recovery problems.

\subsection*{Setting and overview}
We first define the property that a "fingerprint" map $\mathcal F$ applied to a random perturbation of a subgroup $H$ of a group $G$, has a deterministic effect. Consider the following situation:
\begin{itemize}
    \item Let $(G,\cdot)$ be a topological group and $H\subset G$ be a subgroup.
    \item We let $\mathcal{D}$ be a fixed set of probability measures over $G$.
    \item We consider a fixed injective linear $*$-homomorphism $\mathcal{F}:\mathcal M_b(G)\to \mathcal B$ where $\mathcal B$ is a $C^*$-algebra and $\mathcal M_b(G)$ is the space of {\it $G$-shift-bounded measures} on $G$, which are defined as Borel measures $\mu$ on $G$ such that for every compact $K\subset G$ there holds $\sup_{g\in G}|\mu|(g\cdot K)<\infty$, $|\mu|$ being the total mass of $\mu$.
\end{itemize} 
For fixed $\mu\in \mathcal D$ we perturb $H\subset G$ as follows. Consider a family of i.i.d $G$-valued random variables $\xi=(\xi_h)_{h\in H}$, defined over a probability space $(\Omega,\mathcal{T},\mathbb{P})$, with fixed common law $\mathbb P_\xi=\mu$, and consider the realization set 
\begin{equation}\label{realization}
\begin{split}
    H_{\mu}(\omega)&:=\{h\cdot\xi_h(\omega):h\in H\}\subset G\qquad\omega\in\Omega.
    \end{split}
\end{equation}
We then use the abbreviated notation 
\[
    \delta_{H_{\xi}(\omega)}:=\sum_{h\in H_{\xi}(\omega)}\delta_{h},
\]
where $\delta_h$ is the Dirac delta at $h\in H_{\xi}(\omega)$. Below we give the definitions of interest in this paper, within an abstract setting, which will be sepecialized in \Cref{def:recover}.
\begin{defi}\label{def:almostsure}
    With notations and definitions as above, we introduce the following notions:
    \begin{enumerate} [(A)]
        \item For real numbers $\mathsf{err}\geq 0$ and $\delta\in[0,1)$ we say that \emph{the $\mathcal F$-action on $H\subset G$ is $(\mathsf{err},\delta)-$probably approximately determined (PAD) for $\mathcal D-$perturbations} if there exists a (deterministic) function
    \begin{equation}\label{cmu}
        \Phi_H:\mathcal D\to \mathcal B
    \end{equation}
    such that, for all $\mu \in\mathcal D$ for $H_\mu(\omega)$ as in \eqref{realization},
    \begin{equation}\label{cmubis}
        \mathbb P\left\{\left\|\mathcal F(\delta_{H_\mu}) - \Phi_H(\mu)\right\|_{\mathcal B}> \mathsf{err}\right\}\leq\delta.
    \end{equation}
    \item We say that {\em the $\mathcal F$-action on $H\subset G$ is almost-surely determined (ASD) for $\mathcal D-$perturbations} if there exists a (deterministic) function $\Phi$ as in \eqref{cmu} such that for all $\mu\in\mathcal D$ for $H_\mu(\omega)$ as in \eqref{realization} we have
\begin{equation}\label{sep}
    \mathcal{F}(\delta_{H_{\mu}(\omega)})=\Phi_H(\mu),\quad \text{for $\mathbb P$-almost every $\omega$.}
\end{equation}
    \end{enumerate}
\end{defi}

In particular, the notion of $(0,0)-$PAD coincides with the one of ASD.
\medskip

{\bfseries Motivation for our hypotheses on $\mathcal F$.} While working with noise-robustness of fingerprints $\mathcal F$ in high generality may be interesting, some extra structure on $\mathcal F$ is crucial for being able to prove explicit recovery results. This is also true in different settings for signal recovery problems such as \cite{signoise}. For our work, a good reason to work with the $*$-morphism hypothesis on $\mathcal F$ is that it actually allows to find explicit formulas on $\Phi_H$ and to prove at the same time {\em determination} to {\em recovery properties} (see \Cref{def:recover} below). Rather than working with general $*$-morphisms, we chose to restrict to the case of Fourier transforms in this work as a first step, and it is an interesting direction of future research to extend the results to more general $\mathcal F$.  

\medskip
{\bfseries Explicit formulas for $\Phi_H$.} In case $G, H$ are finite groups, it seems to be more often the case that only $(\mathsf{err},\delta)-$PAD properties hold. But it is instructive to keep in mind that if almost-sure deterministic action \eqref{sep} holds for this case, then the underlying deterministic $\Phi_H$ can be made more explicit. Indeed, by taking the expected value in \eqref{sep} we get the following:
\begin{equation}\label{eq:phihfinite}
    \Phi_H(\mu)=\mathbb E\left[\mathcal F(\delta_{H_\mu})\right] = \mathcal F\left(\mathbb E[\delta_{H_\mu}]\right) = \mathcal F(\delta_H * \mu)=\mathcal{F}(\delta_H)\mathcal{F}(\mu).
\end{equation}
Above, we interpret $\delta_{H_\mu}$ as a random measure, we used linearity of expectation and of $\mathcal F$, and the fact that $\mathcal F$ is a $*$-morphism from finite measures on $G$ with convolution product to $\mathcal{B}$. 
\medskip

For the case of infinite $G$ and $H$, we can prove almost-sure recoverability via ergodicity of $H$-action in combination with a law of large numbers. However in general the manipulations as in the finite case are not warranted, and due to this, we focus on specific $\mathcal F$ defined as the Fourier transform on translation-bounded measures. To define $\mathcal F$ we fix a distance compatible with the metric on $G$ and let $B_R$ be the ball of radius $R$ centered at the identity. Then, for fixed $X\subset G$, we define
\begin{equation}\label{eq:fourierdef}
    \mathcal F(\delta_X)(\pi) :=\lim_{R\to\infty}\frac1{f_X(R)}\sum_{x\in X\cap B_R}\pi(x),    
\end{equation}
where $f_X(R)\sim |X\cap B_R|$ denotes the asymptotic ball growth in $X$. If $H\subset G$ is a discrete infinite group or a lattice in a nilpotent $G$, then $f_H(R)=C_H R^k$, in which $k$ is the dimension of $H$. Then we consider the case of tame $\mu$, which have the moment condition $\mathbb E_\mu[|\xi|^{k+\epsilon}]<\infty$ for some positive $\epsilon$. Then for $\mathbb P$-almost all $\omega\in\Omega$ we may take $f_{H_\mu(\omega)}=f_H$.

In the above case, we can build an explicit description of the function $\Phi$ in \Cref{def:almostsure} (B). Indeed, assuming \eqref{sep} holds, we get, for fixed $\mu\in\mathcal D$ and $\pi\in \widehat G$:
\begin{eqnarray}
    \Phi_H(\mu)(\pi)&=&\mathbb E\left[\mathcal F(\delta_{H_\mu})(\pi)\right]=\mathbb E\left[\lim_{R\to\infty} \frac1{C_H R^k} \sum_{x\in H_{\mu}(\omega)\cap B_R}\pi(x) \right]\nonumber\\
    &=& \mathbb E\left[\lim_{R\to\infty} \frac 1{C_H R^k}\sum_{h\in H\cap B_R}\pi(h\cdot \xi_h) \right]=\lim_{R\to\infty} \frac 1{C_H R^k}\sum_{h\in H\cap B_R}\pi(h) \mathbb E\left[\pi(\xi_h) \right]\nonumber\\
    &=&\mathcal F(\delta_H)(\pi)\int_G \pi(x)d\mu(x),\label{eq:phihinf}
\end{eqnarray}
with the integral taken in the sense of Bochner. The first equality in the above second line is based again on the moment hypothesis on $\mu$, and will be justified in \Cref{lemmalimxiDFT-nil}. Then we used the fact that unitary representations $\pi$ are group homomorphisms, and the fact that the $\xi_h$ are i.i.d. with law $\mu$.
\medskip

{\bfseries Main Results.} We now specialize to the setting of our main results. In the above paragraphs we have seen that $\Phi_H$ has an explicit expression, with which the property in \Cref{def:almostsure} amounts to an -- approximate in case (A) and almost sure in case (B) -- {\em factorization} property of the noise. If the factors multiplying $\mathcal F(\delta_H)$ on the right in \eqref{eq:phihfinite} or \eqref{eq:phihinf} are invertible, then we can actually recover (approximately or almost-surely) $\mathcal F(\delta_H)$ from a single noisy measurement. Further, if we restrict to $\mathcal F$ as the above Fourier transforms, then we have natural choices of $\mathcal B$ as mentioned in the below definition.

\begin{defi}\label{def:recover}
\begin{enumerate}[(A)]
    \item If $H<G$ are finite groups and $\mathcal F: L^2(G)\to L^2(\widehat G)$ is the normalized Fourier transform, we say that $\mathcal F$ can be $(\mathsf{err},\delta)-$approximately recovered on $H$ from $\mathcal D-$perturbations if case (A) of \Cref{def:almostsure} with holds with $\Phi_H$ as in \eqref{eq:phihfinite}.
    \item If $G$ is an abelian or nilpotent, locally compact, simply connected Lie group, $H<G$ is a $k$-dimensional discrete infinite subgroup, and $\mathcal F:\mathcal M_p(G)\to L^\infty(\widehat G)$ is defined as in \eqref{eq:fourierdef}, we say that $\mathcal F$ can be almost-surely recovered on $H$ from $\mathcal D-$perturbations if case (B) of \Cref{def:almostsure} holds with $\Phi_H$ as in \eqref{eq:phihinf}.
\end{enumerate}
\end{defi}
Then our main results are summarized in the next two theorems.

\begin{thma}\label{Thm A}
Let $G$ be a finite group and $H\subset G$ be a normal subgroup with $|H|\geq 24$. Let $\mathcal F$ be the (normalized) Discrete Fourier Transform of $G$. Let $\mathcal{D}$ be the class of probability measures $\mu$ on $G$ for which $C=C(\mu, G):=\sup_{\pi\in\widehat{G}}\mathsf{Var}_{\xi\sim\mu}[\pi(\xi)]<\infty$, where $\widehat{G}$ is the space of all unitary irreducible representations of $G$. Then for fixed $\delta\in(0,1)$, we have that the laws from $\mathcal{D}$ are $(\mathsf{err},\delta)$-probably approximately recoverable through $\mathcal F$ from perturbations with law in $\mathcal D$, where 
\begin{equation}\label{apperrbound}
\mathsf{err}=4C\frac{|H|^{7/2}}{|G|^2}\left(\sqrt{\frac{3}{\delta}}+\frac{C}{|H|^{1/2}}\right)\max_{\pi\in\widehat{G}}d_{\pi}.
\end{equation}
\end{thma}

Note that $C=C(\mu,G)$ represents the main dependence of our estimates on the geometry of $G$.
\medskip

{\bfseries Possible application cases of Theorem A.} In numerical applications, such as for numerical simulations of materials, or for signal processing via finite codes, amongst others, one most commonly one has to control error bounds for finite discretizations, with special focus on asymptotics and scalability of such bounds. For the question of robustness to random perturbations, \Cref{Thm A} provides such control, and allows to find what relationships between the cardinalities of $H$ and $G$ allow small error in asymptotics of interest. Thus this result is of independent interest, and in this paper we use this theorem for our \Cref{Thm B} below.

\medskip
Our second main result is concerned with extending the known recovery results for lattices in Euclidean spaces to more general groups. We consider the natural case of infinite discrete subgroups of nilpotent Lie groups:
\medskip

\begin{thma}\label{Thm B} 
Let $G$ be a $d$-dimensional lcsc nilpotent Lie group, and assume that $H<G$ is a $k$-dimensional lattice with averaged Fourier transform denoted $\mathcal F_k$. For fixed $\varepsilon>0$, let $\mathcal{D}$ be the set of probability measures on $G$ with finite $k+\varepsilon$ moment. Then $H$ is almost-surely recoverable through $\mathcal{F}_k$ from perturbations with law in $\mathcal D$.
\end{thma}

We obtain \Cref{Thm B} from the result of \Cref{Thm A}, via a suitable discrete to-continuum approximation process. Precisely, let $H<G$ be as in \Cref{Thm B}. In \Cref{sec: eucrecov}, \Cref{sec:heis} and \Cref{ssec:recovnil} below, we prove that there is a sequence of finite groups $H_n\subset G_n$ which converge to $H$ and $G$ respectively, in Gromov-Hausdorff topology, and such that the Fourier transforms $\mathcal{F}(\delta_X)$ and $\mathcal{F}(\delta_{X_{\xi}})$ can be approximated by $\mathcal{F}_{\Gamma_n}(\delta_{H_n})$ and $\mathcal{F}_{\Gamma}(\delta_{(H_n)_{\xi}})$, respectively; for more details see \Cref{DFT-FT}, \Cref{DFT-FT-Heis}, \Cref{lemmalimxiDFT-Heis} and \Cref{nilDFT-FT} below. Thus, from a computational viewpoint, \Cref{Thm A} not only gives us, in the limit, a noise-recovery such as \Cref{Thm B} but also allows us to obtain approximations of the recovery of $\mathcal{F}(\delta_X)$ by $\mathcal{F}_{\Gamma_n}(\delta_{H_n})$ with a control for the error bounds due to its explicit (and computable) form \eqref{apperrbound}.
\medskip

{\bfseries Motivation for the hypothesis that $G$ is nilpotent in Theorem B.} Theorem B describes recovery of discrete subgroups in a class of Lie groups $G$, which is a second important model case besides the result of \Cref{Thm A} for finite $G$. The motivation for the hypotheses of Theorem B is based on the following restriction of our proofs:
\begin{enumerate}
    \item For the good definition of $\mathcal F$ as in \eqref{eq:fourierdef}, we require $X\subset G$ to have polynomial growth.
    \item Our almost-sure recovery proof is based on a version of the central limit theorem to $\xi_p$ for $p\in H$, which ensures \eqref{sep} almost surely. For this, we require $H$ to be infinite.
    \item In order use an averaging property of groups $H$, and to formulate the bound on $k+\epsilon$ moments for the law of $\xi$, we require that $H$ have polynomial growth $|B_H(x,R)|=O(R^k)$.
    \item In order to apply the general result of \Cref{Thm A}, we approximate $G$ via finite groups $G_n$ as $n\to\infty$. We also require that there exist subgroups $H_n<G_n$ converging to $H$.
\end{enumerate}
Due to points 2. and 4. we require that $|G_n|, |H_n|\to\infty$ and due to 1. and 3. we require that $G$ is of polynomial growth. These conditions on the $G_n$, and the fact that $G$ is not discrete, already imply that $G$ should be a finite-dimensional nilpotent Lie group, due to the following known result:
\begin{thm}[Corollary of Thm. 3.2.2 from \cite{scallim}]\label{prop:bft}
    Suppose that $0<r_n\leq R_n$ are two unbounded sequences and consider a sequence $(G_n,d)$ of discrete metric groups with ball growth $|B_{G_n}(x,R_n)|=O(R_n^q)$. Then $(G_n,d/r_n)$ have a subsequence converging for the pointed Gromov-Hausdorff topology to a connected nilpotent Lie group equipped with the left-invariant Carnot-Caratheodory metric.
\end{thm}

{\bfseries Possible application cases of Theorem B.} The special case of $G=\mathbb R^d$ and $H$ a lattice describes the stability of spectral properties of a crystal under perturbations (in case $d\leq 3$ especially), and other directions for application may come from the case that $H$ is interpreted as a \emph{code}, or as the discretization of a signal with values in $G$. This setup for $G=\mathbb R^d$ with large $d$ is the traditional setup for lattice-based cryptography, which covers an important percentage of successful quantum-resistant cryptography algorithms \cite{nistreport}. We point out here that the same applications also may be extended to the more general case that $G$ is a nilpotent group. To see why this may be interesting to develop, it is worth mentioning that Heisenberg groups (in the simplest case) and in fact all nilpotent Lie groups represent the state space of quantum systems in the presence of external magnetic fields (see  e.g. \cite{quan} for an introduction, and formalism for an interpretation for the Engel group in \cite{engel}, as well as \cite{step3} for a broader description). For these general $G$, similarly to $G=\mathbb R^d$ the role of $H$ is that of a possible discretization of signals encoded as quantum states of a system, and \Cref{Thm B} gives a possible way to recover corrupted signals of this type.

\medskip

{\bfseries Some natural open directions.} Note that a direct extension of \Cref{Thm B}, would be to include the case that $G$ is itself a lattice in a nilpotent Lie group $G'$, and $H$ is a sublattice. In this case we can work with \Cref{Thm B} directly in $G'$ itself. A second extension, which seems straightforward but for which we did not find good enough motivations to include it in this work, would be to consider the product of $H,G$ with finite groups $H'<G'$.
\medskip

We now indicate some natural extensions of the result of \Cref{Thm B}, which seem interesting to pursue in future work, but would require fundamental changes in our proof strategy.

\begin{itemize}
    \item Quasicrystals $H$ in a nilpotent Lie group.
    \item Sublattices of infinite lattices in a nilpotent Lie group.
    \item Infinite groups in amenable groups of intermediate or exponential growth.
\end{itemize}

\noindent{\bf Some short words on notation. }
Now we include some basic definitions and notations which we will use throughout this work. We use the notation $|\cdot|$ indistinguishably for the Euclidean norm in $\mathbb{R}^d$, for the module of a complex number, for the Lebesgue measure, and the cardinality of a finite set; this will be clear in each context. As usual we denote by $B(x,R):=\{y\in\mathbb{R}^d: |x-y|<R\}$ the euclidean open ball centered at $x\in\mathbb{R}^d$ with radius $r>0$ and we use the notation $B_R:=B(0,R)$.
\medskip

For a Banach space $\mathcal{B}$, we write $\|\cdot\|_{\mathcal{B}}$ for a norm in $\mathcal{B}$. If $\mathcal{H}$ is a Hilbert space, we write $\langle\cdot,\cdot\rangle_{\mathcal{H}}$ for the inner product in $\mathcal{H}$. Given a measure space $\mathcal{M}$, we use the notation $L^p(\mathcal{M})$ for the Banach space of $p$-integrable functions; in the case when $\mathcal{M}$ is finite, to emphasize its discrete nature we write $\ell^p(\mathcal{M})$.
\medskip 

We usually denote a diffeomorphism between manifolds or an isomorphism between Hilbert spaces by $\cong$. This also will be clear in each situation where we use $\cong$ as an equivalence. Finally, the expression $C_1(M)\ll C_2(N)$ and $\varepsilon\lesssim 1$ will mean that $C_2(N)/C_1(M)\to\infty$ from sufficiently large values of $M,N$, and that $\varepsilon<C$ for some suitable constant $C$, respectively.
\medskip

\noindent{\bf Organization of the paper. }In \Cref{sec: recovfingr}, we consider the problem of quantitative recovery in the most general setting. In \Cref{ssec: repfingr} we review basic material on representations of finite groups and then, in \Cref{ssec: DFT}, we give some basics about the Fourier transform in finite groups, which will be necessary to state our recovery results. The core of this section is \Cref{ssection: nonabelFT}, where we prove \Cref{Thm A} (see \Cref{basicthmH}).
\medskip

In \Cref{sec: recovgennil} and \Cref{ssec:recovnil}, we are dedicated to proving \Cref{Thm B}. We provide some necessary facts about nilpotent Lie groups, induced representations, and Kirillov's orbit method as machinery to obtain irreducible representations in these Lie groups. Then we state and prove the main result in this section, namely, \Cref{nilrecov}.
\medskip

Finally in \Cref{sec: examples} we provide examples of the above in the most canonical abelian and non-abelian instances, namely, the Euclidean space and the Heisenberg group, where some explicit computations are done. 
\medskip
%%%%%%%%%%%%%%%%%%%%%%%%%%%%%%%%%%%%%%%%%%%%%%%%%%%%%%%%%%%%%%%%%%%%%%%%%%%%%%%%%%%%%%%%%%%%%%%%%%%%%%%%%%%%%%%%%%%%%%%%%%%%%%%%%%%%%%%%%%%%%%%%

%%%%%%%%%%%%%%%%%%%%%%%%%%%%%%%%%%%%%%%%%%%%%%%%%%%%%%%%%%%%%%%%%%%%%%%%%%%%%%%%%%%%%%%%%%%%%%%%%%%%%%%%%%%%%%%%%%%%%%%%%%%%%%%%%%%%%%%%
\section{Approximate recovery in finite groups}\label{sec: recovfingr}

\subsection{Background on Representations of Finite Groups}\label{ssec: repfingr}

This section establishes some basic terminology and the necessary machinery of Fourier Analysis in Finite (non-Abelian) Groups. Given a finite-dimensional vector space $V$ over $\C$, define by $GL(V)$ be the vector space of all the linear invertible linear transformations $T:V\to V$; this space is naturally identified with the space of invertible matrices $GL(n,\C)$, where $n=\mathsf{dim}(V)$. Throughout this section, $\Gamma$ will denote a finite group. A homomorphism $\pi:\Gamma\to GL(V)$ is called a (finite dimensional) {\em representation} of $\Gamma$. Thus, for every $g\in \Gamma$, $\pi(g)$ is a matrix whose complex entries are denoted by $\pi_{ij}(g)$. The dimension of $V$ is denoted by $d_{\pi}$ and is called the {\em degree} (or {\em dimension}) of the representation $\pi$. If $\pi$ takes values into the set of unitary operators $\mathcal{U}(V)\cong \mathcal{U}(n,\C):=\{A\in GL(n,\C):\ AA^*=I\}$, then $\pi$ is called a {\em unitary} representation of $\Gamma$; in particular the Hermitian inner product $\langle u,v\rangle:=uv^t,\ u,v\in\C^n$ is preserved under unitary transformations.

\medskip

The building blocks of Fourier Analysis in (nonabelian) groups are the irreducible representations. Given a representation $\pi:\Gamma\to GL(V)$ and a subspace $W\leq V$ which is $\pi$-invariant, i.e, $\pi(g)W\subset W$ for every $g\in \Gamma$, the map $\pi':\Gamma\to GL(W)$ given by $\pi'(g):=\pi(g)|_W$ is called a {\em subrepresentation} of $\pi$. A representation $\pi$ is called {\em irreducible} if its only subrepresentations are the trivial one and $\pi$ itself.

\medskip

We say that two representations $\pi_1:\Gamma\to GL(V_{\pi_1})$ and $\pi_2:\Gamma\to GL(V_{\pi_2})$ are {\em equivalent} if there is a isomorphism $T:V_{\pi_1}\to V_{\pi_2}$ such that $T\pi_1(g)=\pi_2(g)T$ for every $g\in \Gamma$; we denote this equivalence relation by $\pi_1\cong\pi_2$. The following Proposition says that every unitary representation can be decomposed into irreducible components (see Proposition \cite[Chapter 15]{terras}).

\medskip

\begin{prop}\label{irredecomp}
\begin{enumerate}
    \item Every representation is equivalent to a unitary representation.
    
    \item Every unitary representation $\pi: \Gamma\to \mathcal{U}(V_\pi)$ is equivalent to

\begin{equation*}
    \pi_1\oplus\pi_2\oplus\ldots\oplus\pi_r:=\begin{pmatrix}
\pi_1 &  & 0\\
 & \ddots & \\
0 &  & \pi_r
\end{pmatrix}
\end{equation*}

where each block $\pi_j$ is an irreducible subrepresentation of $\pi$. 
\end{enumerate}
\end{prop}

Given two representations $\pi:\Gamma\to GL(V_{\pi})$ and $\rho:\Gamma\to GL(V_\rho)$, consider the space $I(\pi,\rho)$ of {\em intertwining operators} defined by

\[
I(\pi,\rho):=\{T:V_\pi\to V_\rho: \text{T is linear and }T\circ\pi(g)=\rho(g)\circ T, \text{ for every }g\in \Gamma\};
\]

In particular, $I(\pi):=I(\pi,\pi)$ is the set of linear operators $T:V_\pi\to V_\pi$ which commute (under composition) with $\pi(g)$ for every $g\in \Gamma$. The condition that two representations $\pi,\rho$ are equivalent can now be reformulated by requiring that $I(\pi,\rho)$ contains an isomorphism.

The following Lemma is one of the cornerstones in Fourier Analysis in non-abelian groups; we refer to \cite[Chapter 15]{terras} for further details.

\begin{lemma}[Schur's Lemma]\label{schur}
\begin{enumerate}
    \item A unitary representation $\pi$ of $\Gamma$ is irreducible if and only if $I(\pi)$ contains only scalars multiples of the identity.
    
    \item If $\pi$ and $\rho$ are unitary irreducible representations of $\Gamma$, then
    
    \[
    \mathsf{dim}(I(\pi,\rho))= \left\{ \begin{array}{lc}
             1 & \text{if }\pi\text{ and }\rho\text{ are equivalent,} \\
             \\ 0 &  \text{otherwise.}
             \end{array}
   \right.
    \]
\end{enumerate}
\end{lemma}
By identifying representations under equivalence, we define the {\em dual space} $\widehat{\Gamma}$ of $\Gamma$ as the set of equivalence classes of irreducible unitary representations of $\Gamma$, i.e

\[
\widehat{\Gamma}:=\{[\pi]:\ \pi:\Gamma\to U(V_\pi)\text{ is an irreducible unitary representation of }\Gamma\},
\]
where $[\pi]$ denotes the equivalence class of $\pi$; we simply write $\pi\in\widehat{\Gamma}$ when there is no confusion about the notation. It is well-known that

\begin{equation}\label{dimrep}
    \sum_{\pi\in\widehat{\Gamma}}d_{\pi}^2=|\Gamma|.
\end{equation}

For a given $\pi\in\widehat{\Gamma}$ and a unitary representation $\rho$ of $\Gamma$, we denote by $m(\pi,\rho)$ be the number of times that $\pi$ occurs in the decomposition in irreducible components of $\rho$ given in \Cref{irredecomp}.
\medskip

The entries $\pi_{ij}$ of a representation $\pi:\Gamma\to GL(V)$ can be regarded as elements of $L^2(\Gamma)$ with the usual inner product 

\[
\langle \phi, \psi\rangle_{L^2(\Gamma)}:=\sum_{g\in \Gamma}\phi(g)\overline{\psi(g)}\qquad\text{ for every }\phi,\psi\in L^2(\Gamma).
\]
The following theorem establishes orthogonality relations for the entries of irreducible unitary representations (see Theorem 1 in \cite[Chapter 15]{terras}).

\begin{thm}[The Schur Orthogonality relations]\label{Schurortho}
If $\pi,\rho$ are two inequivalent irreducible finite dimensional unitary representations of $\Gamma$, then

\[
\langle\pi_{ij},\rho_{rs}\rangle=0\qquad\text{ for all }i,j,r,s.
\]

Moreover, if $\delta_{ij}$ denotes the Kronecker delta function, there holds

\[
\langle\pi_{ij},\pi_{rs}\rangle=\frac{|\Gamma|}{d_{\pi}}\delta_{ij}\delta_{rs}.
\]
\end{thm}

\begin{thm}[Clifford's Theorem]\label{cliffordthm}
Suppose that $H$ is a normal subgroup of $\Gamma$ and let $\pi\in\widehat{\Gamma}$. Denote by $\pi|_H$ be the restriction of $\pi$ to $H$ and consider its decomposition in irreducible components

\begin{equation}\label{cliffordecom}
    \pi|_H\cong \pi_1\oplus\ldots\oplus\pi_k\qquad\text{where }\pi_i\in\widehat{H}\text{ for every }1\leq i\leq k.
\end{equation}

Then for each $1\le j\le k$ there exists $g_j\in G$ such that for all $h\in H$ we have $\pi_j(h)=\pi_1(g_j^{-1}hg_j)$. Moreover, the following hold:
\begin{enumerate}
   \item All the irreducible components in \eqref{cliffordecom} have the same degree, i.e, for every $1\leq i\leq k$ there holds that $d_{\pi_1}=d_{\pi_i}$.
   
   \item All the irreducible components in \eqref{cliffordecom} have the same multiplicity, i.e, for every $1\leq i\leq k$ there holds that $m(\pi_1,\pi|_H)=m(\pi_i,\pi|_H)$.
\end{enumerate}
\end{thm}

\subsection{Fourier Transform on Finite groups}\label{ssec: DFT}

Define the Fourier Transform of a function $f\in L^1(\Gamma)$ by 
\begin{equation}\label{FTrep}
    \mathcal{F}(f)(\pi):=\frac{1}{|\Gamma|}\sum_{g\in \Gamma}f(g)\pi(g),\qquad\text{ for every }\pi\in\widehat{\Gamma}.
\end{equation}
\medskip

For $\phi,\psi\in L^2(\Gamma)$ define the convolution $\phi*\psi$ as follows
\[
    \phi*\psi(g):=\sum_{h\in \Gamma}\phi(h)\psi(h^{-1}g).
\]
Then with the above definitions, we get that
\begin{equation}\label{nonabelconv}
    \mathcal{F}(\phi*\psi)(\pi)=|\Gamma|\mathcal{F}(\phi)(\pi)\mathcal{F}(\psi)(\pi).
\end{equation}
Denote by $L^2(\widehat{\Gamma})$ be the vector space

\[
L^2(\widehat{\Gamma}):=\left\{f:\widehat{\Gamma}\to\bigsqcup_{\pi\in\widehat{\Gamma}}\mathbb{C}^{d_{\pi}\times d_{\pi}}:\ f(\pi)\in\mathbb{C}^{d_{\pi}\times d_{\pi}}\text{ for every }\pi\in\widehat{\Gamma}\right\}. 
\]

We endow $L^2(\widehat{\Gamma})$ with the inner product $\langle\cdot,\cdot\rangle_{L^2(\widehat{\Gamma})}$ given by
\begin{equation}\label{innerprodual}
    \langle f,g\rangle_{L^2(\widehat{\Gamma})}:=\frac{1}{|\Gamma|}\sum_{\pi\in\widehat{\Gamma}}d_{\pi}\mathsf{Tr}(f(\pi)g(\pi)^*).
\end{equation}
Here $\mathsf{Tr}(T)$ is the {\em trace} of a linear operator $T:V\to V$, and is defined by $\mathsf{Tr}(T):=\sum_{j=1}^{d_{\pi}}\langle Te_j,e_j\rangle$, where $(e_i)$ ranges over any orthonormal basis of $V$. In particular, for all $f\in L^2(\widehat{\Gamma})$ there holds that
\begin{equation}\label{L2hatgamma}
    \|f\|_{L^2({\widehat{\Gamma}})}^2=\frac{1}{|\Gamma|}\sum_{\pi\in\widehat{\Gamma}}d_{\pi}\sum_{i,j=1}^{d_{\pi}}|f_{ij}(\pi)|^2,
\end{equation}
where $f_{ij}(\pi)$, $i,j=1,\ldots,d_{\pi}$ are the entries of the matrix $f(\pi)$ in a basis $(e_i^{(\pi)})_{i=1}^{d_{\pi}}$ of $V_{\pi}$.
\medskip

By the {\em Plancherel Theorem}, we have that the Fourier Transform is an isometric isomorphism between $L^2(\Gamma)$ and $L^2(\widehat{\Gamma})$, i.e., that for any $f,g\in L^2(\Gamma)$ there holds
\begin{equation}\label{plancherel}
    |\Gamma|^2\langle\mathcal{F}(f),\mathcal{F}(g)\rangle_{L^2(\widehat{\Gamma})}=\langle f,g\rangle_{L^2(\Gamma)}.
\end{equation}
By the {\em Peter-Weyl Theorem}, the matrix entries $\pi_{ij}\in L^2(\Gamma)$ of $\pi\in\widehat{\Gamma}$ form a complete orthogonal set in $L^2(\Gamma)$. In particular, \eqref{plancherel} can be written as
\begin{equation}\label{plancherel2}
    \langle f,g\rangle_{L^2(\Gamma)}=\frac{1}{|\Gamma|}\sum_{\pi\in\widehat{\Gamma}}d_{\pi}\sum_{i,j=1}^{d_{\pi}}\langle f,\pi_{ij}\rangle_{L^2(\Gamma)}\overline{\langle g,\pi_{ij}\rangle}_{L^2(\Gamma)}.
\end{equation}
Finally, the {\em Fourier inversion formula} says that every $f\in L^2(\Gamma)$ can be reconstructed from its Fourier transform as follows
\begin{equation}\label{Finv}
    f(h)=\frac{1}{|\Gamma|}\sum_{\pi\in\widehat{\Gamma}}d_{\pi}\mathsf{Tr}(\pi(h^{-1})\mathcal{F}(\pi))=\frac{1}{|\Gamma|}\sum_{\pi\in\widehat{\Gamma}}d_{\pi}\sum_{i,j=1}^{d_\pi}\langle f,\pi_{ij}\rangle\pi_{ij}(h).
\end{equation}
\subsection{Estimating the Fourier Transform under random perturbations}\label{ssection: nonabelFT}

Let $\Gamma$ be a non-simple finite group and $H$ a normal subgroup of $\Gamma$. Recall that the Fourier Transform of a function $f\in L^2(\Gamma)$ by
\begin{equation}
    \mathcal{F}(f)(\pi)=\widehat{f}(\pi):=\frac{1}{|\Gamma|}\sum_{g\in\Gamma}f(g)\pi(g),\qquad\text{for }\pi\in\widehat{\Gamma}. 
\end{equation}
Consider the Dirac delta of $H$ defined by 
\begin{equation*}
    \delta_{H}(g):=\left\{ \begin{array}{lc}
             1 & \text{if }g\in H \\
             \\ 0 &  \text{otherwise.}
             \end{array}
   \right.
\end{equation*}
As in the Euclidean case, our first goal is to relate the Fourier Transform of $\delta_{H}\in L^2(\Gamma)$ with the annihilator set $H^{\perp}$ defined by
\begin{equation}
    H^{\perp}:=\{\pi\in\widehat{\Gamma}:\ \pi(h)=\mathsf{ld}_{V_{\pi}}\text{ for every }h\in H\}
\end{equation}
We also have the following characterization, which must be a classical result, but for which we were not able to find an explicit reference:
\begin{lemma}\label{Hperpq}
$H^{\perp}\cong\widehat{\Gamma/H}$. 
\end{lemma}

\begin{proof}[Sketch of the proof. ]Let $i:\Gamma\to \Gamma/H$ be the canonical map associating to every $g\in\Gamma$ its equivalence class $i(g)\in\Gamma/H$. The natural map $\phi:H^{\perp}\to\widehat{\Gamma/H}$ given by 
\[
\phi(\pi)(i(g)):=\pi(g)
\]
is well-defined and determines a group isomorphism between $H^{\perp}$ and $\widehat{\Gamma/H}$. We leave the details to the reader.
\end{proof}

For the next discussion, we will study the unitary representations $\pi\not\in H^\perp$. Given a representation $\pi\in\widehat{\Gamma}\setminus H^{\perp}$, by \Cref{irredecomp} we can decompose $\pi|_H$ (modulo equivalence) in irreducible components

\[
\pi|_H=\pi_1\oplus\ldots\oplus\pi_k,\qquad \text{ where }\pi_i\in\widehat{H}\text{ for every }i=1,\ldots,k.
\]

In this setting, consider the matrix $\widehat{\delta_H}(\pi)_i$ as the the $d_i\times d_i$-block of $\widehat{\delta_H}(\pi)$ defined by

\begin{equation}\label{deltaproj}
 \left(\widehat{\delta_H}(\pi)\right)_i:=P_i^T\widehat{\delta_H}(\pi)P_i=\frac{1}{|\Gamma|}\sum_{h\in H}\pi_i(h),
\end{equation}

where $P_i$ is the projection onto the subspace of $V_i$ where $\pi_i$ acts, and $d_i$ is the degree of $\pi_i$. Note that in the above definition, we use an abuse of notation: we have $\widehat{\delta_H}\in L^2(\widehat{\Gamma})$ and in general \Cref{irredecomp} cannot be applied to $\widehat{\delta_H}(\pi)$ directly.

\begin{lemma}\label{FTdeltaH}
We have that

\begin{enumerate}
    \item If $\pi\in H^{\perp}$, then $\widehat{\delta_{H}}(\pi)=\frac{|H|}{|\Gamma|}\mathsf{Id}_{V_{\pi}}$. 
    \item If $\pi\in \widehat{\Gamma}\setminus H^{\perp}$, then when we decompose $\widehat{\delta_H}(\pi)$ as in \eqref{deltaproj}, for all $i=1,\ldots,k$ for which $\pi_i\neq\mathsf{Id}_{
    V_{\pi_i}}$, the block $\left(\widehat{\delta_H}(\pi)\right)_i$ vanishes.
\end{enumerate}

In particular, if $H$ is a normal subgroup of $\Gamma$, then

\begin{equation*}
    \widehat{\delta_{H}}(\pi)=\frac{|H|}{|\Gamma|}\mathsf{Id}_{d_{\pi}}\delta_{H^\perp}(\pi)=\left\{ \begin{array}{lc}
             \frac{|H|}{|\Gamma|}\mathsf{Id}_{d_{\pi}} & \text{if }\pi\in H^{\perp} \\
             \\ 0 &  \text{otherwise.}
             \end{array}
   \right.
\end{equation*}

\end{lemma}

\begin{proof}
Firstly, by the definition of the Fourier Transform, we get that

\begin{equation}
    \widehat{\delta_H}(\pi)=\frac{1}{|\Gamma|}\sum_{h\in H}\pi(h).
\end{equation}

If $\pi\in H^{\perp}$ then it is direct that $\widehat{\delta_H}(\pi)=\frac{|H|}{|\Gamma|}\mathsf{ld}_{V_{\pi}}$. 
Thus we concentrate on the case $\pi\not\in H^{\perp}$. In this case, there must exists $h'\in H$ such that $\pi(h')\neq \mathsf{Id}_{V_{\pi}}$. For this choice of $h'$, there holds 

\begin{equation}\label{commutdelta}
    \begin{split}
        \pi(h')\widehat{\delta_H}(\pi)=\frac{1}{|\Gamma|}\sum_{h\in H}\pi(h'h)=\widehat{\delta_H}(\pi)=\widehat{\delta_H}(\pi)\pi(h').
    \end{split}
\end{equation}

In particular for each block in the decomposition of $\pi|_H$ from \Cref{irredecomp}, we have that the operator $\widehat{\delta_H}(\pi)_i$ belongs to $I(\pi_i)$ (if $\pi(h')=\mathsf{Id}_{V_{\pi}}$) then $\widehat{\delta_H}(\pi)_i$ obviously commutes with $\pi_i(h')$). By Schur's Lemma, there is a scalar $c_i$ such that $\widehat{\delta_H}(\pi)_i=c_i\cdot\mathsf{Id}_{V_{\pi_i}}$. If $1\leq i\leq k$ satisfies that $\pi_i(h')\neq\mathsf{Id}_{V_{\pi_i}}$, then from \eqref{commutdelta} we get that $c_i=0$, otherwise $\pi_i(h')=\mathsf{Id}_{V_{\pi_i}}$ contradicting the choice of $\pi$ and $h'\in H$. Therefore the matrix $\widehat{\delta_H}(\pi)_i$ vanishes as claimed.
\medskip

Now, assume that $H$ is a normal subgroup of $\Gamma$. In this case, consider the matrix

\[
P_{\pi}:=\sum_{g\in\Gamma}\pi(g)\widehat{\delta_H}(\pi)\pi(g^{-1}).
\]

There holds $P_{\pi}\in I(\pi)$, since for every $g'\in\Gamma$ we can write

\begin{equation*}
    \begin{split}
        \pi(g')P_{\pi}&=\sum_{g\in\Gamma}\pi(g'g)\widehat{\delta_H}(\pi)\pi(g^{-1})=\sum_{p\in\Gamma}\pi(p)\widehat{\delta_H}(\pi)\pi(p^{-1}g')=P_{\pi}\pi(g').
    \end{split}
\end{equation*}

Hence, by Schur's Lemma, there is a constant $\lambda$ such that

\begin{equation}\label{Pconstant}
P=P_{\pi}=\lambda\cdot\mathsf{Id}_{d_{\pi}}.
\end{equation}

On the other hand, since $H$ is a normal subgroup of $\Gamma$, we have

\begin{equation}\label{Pdelta}
P_{\pi}=\frac{1}{|\Gamma|}\sum_{g\in\Gamma}\pi(g)\left(\sum_{h\in H}\pi(h)\right)\pi(g^{-1})=\frac{1}{|\Gamma|}\sum_{g\in\Gamma}\sum_{h\in H}\pi(ghg^{-1})=\sum_{g\in\Gamma}\widehat{\delta_{gHg^{-1}}}(\pi)=|\Gamma|\widehat{\delta_H}(\pi).
\end{equation}

Thus, from \eqref{Pconstant} and \eqref{Pdelta} we get that

\begin{equation}\label{deltaid}
    \widehat{\delta_H}(\pi)=\frac{\lambda}{|\Gamma|}\mathsf{Id}_{d_{\pi}}.
\end{equation}
For $\pi\in H^\perp$, by applying the above to the the identity of $\Gamma$, we obtain that $\lambda=|H|$. For $\pi\in \widehat\Gamma\setminus H^{\perp}$, we have that the values $c_i$ from the paragraph following \eqref{deltaid} are all equal to $c_i=\frac{\lambda}{|\Gamma|}$, and therefore $\lambda=0$. This finishes the proof.
\end{proof}

\begin{remark}
The last formula of \Cref{FTdeltaH} is a consequence of a more general result, which says that $f\in L^2(\Gamma)$ is a class function (i.e., if $f(ghg^{-1})=f(h)$ for every $h,g\in\Gamma$) if and only if $\widehat{f}(\pi)$ is a scalar multiple of $\mathsf{Id}_{d_{\pi}}$. See for instance Tao's Lecture notes of Representation Theory \cite{Tao}. 
\end{remark}

Now, to describe random perturbations of a subgroup $H$ of a group $\Gamma$ as above, let $(\xi_h)_{h\in H}$ be independent identically distributed random vectors defined in a probability space $(\Omega,\mathcal{T},\mathbb{P})$, where for each $h\in H$ the random variable $\xi_h$ is $\Gamma$-valued. Like in the euclidean case, consider the realization set $H_{\xi}$ as the subset of $\Gamma$ defined by

\begin{equation}\label{randomH}
    H_{\xi}:=\{h\xi_h:\ h\in H\}. 
\end{equation}
Equivalently, if $\mu_e$ denotes the law of $\xi_{\mathsf{id}}$, then the law of $h\xi_h$ is the image measure of $\mu_{e}$ under the $H$-left translations $L_h(h'):=hh'$.

\medskip

The main result of this section is an approximate recovery of $\widehat{\delta_H}$ from $\widehat{\delta_{H_{\xi}}}$ as follows.

\begin{thm}\label{basicthmH}
Let $L>3$ and let $H$ be a normal subgroup of $\Gamma$ with cardinality $|H|$ at least $6L/(L-3)$. Consider $\xi_h, h\in H$ be i.i.d. $\Gamma$-valued random variables with $\xi_h\sim\xi$ and $C_{\xi,\Gamma}:=\sup_{\pi\in\widehat{\Gamma}}\mathsf{Var}[\pi(\xi)]<\infty$, and assume that $\varepsilon>\left(\frac{12C^2L}{|H|}\right)^{1/2}$. Then with probability at least $1-\frac{12C^2L}{\varepsilon^2|H|}$ we have
\begin{equation}\label{mainerrorH}
    \left\|\widehat{\delta_{H_{\xi}}}-\widehat{\delta_H}\cdot\mathbb E[\xi(\cdot)]\right\|_{L^2(\widehat{\Gamma})}^2\leq \frac{|H|^4}{|\Gamma|^2}\left(\varepsilon+\frac{4C^2}{|H|}\right)\max_{\pi\in\widehat\Gamma}d_\pi,
\end{equation}
where $\xi(\cdot):\widehat{\Gamma}\to \bigsqcup_{\pi\in\widehat{\Gamma}}\mathbb{C}^{d_{\pi}\times d_{\pi}}$ stands for the map $\xi(\pi):=\pi(\xi)$.
\end{thm}
By the definition of $|\cdot|_{L^2(\widehat{\Gamma})}$ fixed in \eqref{innerprodual}, the left-hand side of \eqref{mainerrorH} can be written as
\begin{equation}\label{1streformG}
\begin{split}
    \left\|\widehat{\delta_{H_{\xi}}}-\widehat{\delta_H}\cdot\mathbb E[\xi(\cdot)]\right\|^2_{L^2(\widehat{\Gamma})}&=\frac{1}{|\Gamma|}\sum_{\pi\in\widehat{\Gamma}}d_{\pi}\sum_{j=1}^{d_{\pi}}\left\|\left(\widehat{\delta_{H_{\xi}}}(\pi)-\widehat{\delta_H}(\pi)\mathbb{E}(\xi(\pi))\right)e_{j}^{(\pi)}\right\|_{V_{\pi}}^2\\
    &=\frac{1}{|\Gamma|^3}\sum_{\pi\in\widehat{\Gamma}}d_{\pi}\sum_{i,j=1}^{d_{\pi}}\left|\sum_{l=1}^{d_{\pi}}\sum_{h\in H}\pi_{il}(h)\left(\pi_{lj}(\xi_h)-\mathbb{E}[\pi_{lj}(\xi)]\right)\right|^2\\
    &=\frac{1}{|\Gamma|^3}\sum_{\pi\in\widehat{\Gamma}}d_{\pi}\sum_{i,j=1}^{d_{\pi}}\left|\sum_{l=1}^{d_{\pi}}\langle\pi_{il},\overline{\pi_{lj}(\xi_h)-\mathbb{E}[\pi_{lj}(\xi)]}\rangle_{L^2(H)}\right|^2,
\end{split}    
\end{equation}
and where $\{e_1^{(\pi)},\dots,e_{d_{\pi}}^{(\pi)}\}$ is an orthonormal basis for $V_{\pi}$. Now for fixed $\pi\in L^2(\widehat{\Gamma})$ we define $f_{il}, g_{lj}\in L^2(H)$ and their extensions by zero $F_{il}, G_{lj}\in L^2(\Gamma)$ by
\[
f_{il}^{(\pi)}(h)=f_{il}(h):=\pi_{il}(h)\qquad\text{and}\qquad g_{lj}^{(\pi)}(h)=g_{lj}(h):=\overline{\pi_{lj}(\xi_h)-\mathbb{E}[\pi_{lj}(\xi)]},
\]
\[
F_{il}^{(\pi)}=F_{il}=f_{il}\delta_H\qquad\text{and}\qquad G_{lj}^{(\pi)}=G_{lj}=g_{lj}\delta_H.
\]
With this notation, our main goal for estimating \eqref{1streformG} is to find a suitable bound for $\sum_{l=1}^{d_{\pi}}\langle f_{il},g_{lj}\rangle_{L^2(H)}$, possibly depending on $i,j$, since \eqref{1streformG} can be rewritten as: 
\begin{equation}\label{estFT}
     \left\|\widehat{\delta_{H_{\xi}}}-\widehat{\delta_H}\cdot\mathbb E[\xi(\cdot)]\right\|^2_{L^2(\widehat{\Gamma})}=\frac{1}{|\Gamma|^3}\sum_{\pi\in\widehat{\Gamma}}d_{\pi}\sum_{i,j=1}^{d_{\pi}}\left|\sum_{l=1}^{d_{\pi}}\langle f_{il},g_{lj}\rangle_{L^2(H)}\right|^2.
\end{equation}
Observe that $f_{il}$ is deterministic and $g_{lj}$ is random.

\medskip

By triangle inequality, Cauchy-Schwarz inequality and the Plancherel formula \eqref{plancherel}, we rewrite the summands in \eqref{estFT} via the Fourier transforms of $F_{il}, G_{lj}$ (compare with \eqref{parseval}):
\begin{equation}\label{nonabelest}
\begin{split}
    \left|\sum_{l=1}^{d_{\pi}}\langle F_{il},G_{lj}\rangle_{L^2(\Gamma)}\right|&=|\Gamma|^2\left|\sum_{l=1}^{d_{\pi}}\left\langle\widehat{F_{il}},\widehat{G_{lj}}\right\rangle_{L^2(\widehat{\Gamma})}\right|\leq|\Gamma|^2\sum_{l=1}^{d_{\pi}}\left\|\widehat{F_{il}}\right\|_{L^2(\widehat{\Gamma})}\left\|\widehat{G_{lj}}\right\|_{L^2(\widehat{\Gamma})}.
    \end{split}    
\end{equation}
We concentrate first on determining $\|\widehat{F_{il}}\|_{L^2(\Gamma)}$. We now fix our orthonormal basis of $V_\pi$ so that it is adapted to the decomposition
\[
\pi|_H=\pi_1\oplus\ldots\oplus\pi_{l(\pi)},\qquad\text{where }\pi_k\in \widehat{H} \text{ for every }k=1,\ldots,l(\pi).
\]
Therefore $\pi_{il}\delta_H\in L^2({H})$ appears as a matrix-entry of some irreducible component, whose index we denote $k(i,l)=k, 1\leq k\leq l(\pi)$; we denote by $\pi_{pq}^{(k)}$ the matrix entry of $\pi_k$ corresponding to $\pi_{il}\delta_H$.
\begin{lemma}\label{suppFil}
As before, let $\pi\in\widehat\Gamma$, let $\pi_{il}$ be a matrix entry of $\pi$ such that $\pi_{il}$ is not constantly zero over $H$, and let $\pi_{k(i,l)}=\pi_k\in\widehat{H}$ be the irreducible component of $\pi|_H$ in which $\pi_{il}$ appears. Then the following holds: 
\begin{enumerate}
    \item The support of $\widehat{F_{il}}$ is the set of all irreducible unitary representations $\rho\in\widehat{\Gamma}$ for which $\pi_k$ occurs in the decomposition in irreducible components of $\rho|_H$.
    \item If $m(\pi_k,\rho|_H)$ is the number of times that $\pi_k$ appears in the decomposition of $\rho|_H$ in irreducible components, then
    \begin{equation}\label{Fil}
        \|\widehat{F_{il}}\|_{L^2(\Gamma)}^2=\frac{|H|^2}{d_{\pi_k}^2|\Gamma|^3}\sum_{\rho\in\widehat\Gamma}d_{\rho}\cdot m(\pi_k,\rho|_H)%=\frac{|H|^3}{d_{\pi_k}^2|\Gamma|^3}\sum_{\rho\in\widehat\Gamma}d_{\rho}\langle\chi_{\pi_k},\chi_{\rho|_H}\rangle_{L^2(H)},
    \end{equation}
    where $\chi_{\rho}(h):=\mathsf{Tr}(\rho(h))$ is the character of the representation $\rho:H\to GL(V_{\rho})$.
\end{enumerate}
\end{lemma}
\begin{proof}[Proof point 1. of \Cref{suppFil}] Let $\rho\in\widehat{\Gamma}$. By \Cref{irredecomp} $\rho|_H$ decomposes as
\[
    \rho|_H=\rho_1\oplus\ldots\oplus\rho_{l(\rho)},\qquad\text{where } \rho_k\in \widehat{H}\text{ for every }1\leq k\leq l(\rho).
\]
With these notations, we will prove that
\[
    \widehat{F_{il}}(\rho)\neq 0\qquad\text{if and only if}\qquad \exists\ 1\leq m\leq l(\rho)\text { such that }\rho_m=\pi_k.
\]
With respect to bases of $V_\pi, V_\rho$ adapted to the decompositions of $\pi|_H, \rho|_H$, we have
\begin{equation*}
    \widehat{F_{il}}(\rho)=\frac{1}{|\Gamma|}\sum_{h\in H}\pi_{il}(h)\rho(h)=\frac{1}{|\Gamma|}\sum_{h\in H}\pi_{pq}^{(k)}(h)\rho(h)=\left(\frac{1}{|\Gamma|}\langle\pi_{pq}^{(k)},\overline{\rho_{rs}}\rangle_{L^2(H)}\right)_{r,s=1}^{d_{\rho}}.
\end{equation*}
If a component $\rho_m$ of $\rho|_H$ is not equivalent to $\pi_k$, then Schur's orthogonality relations (\Cref{Schurortho}) give 
\begin{equation}\label{pirhol2}
\langle\pi_{pq}^{(k)},\overline{\rho_{rs}^{(m)}}\rangle_{L^2(H)}=0\qquad\text{for every entry $\rho_{rs}^{(m)}$ of $\rho_m$}.
\end{equation}
In particular, if this is true for all $m$, we get
\[
    \left(\langle\pi_{pq}^{(k)},\overline{\rho_{rs}}\rangle_{L^2(H)}\right)_{r,s=1}^{d_{\rho}}=0\qquad\text{for every }1\leq r,s\leq d_{\rho}. 
\]
This shows one direction of the claimed implication. Assume now that there is $1\leq m\leq l(\rho)$ such that $\rho_m=\pi_k$. Then again, by \Cref{Schurortho} up to a change of basis of $V_\rho$ we get
\[
    \langle\pi_{pq}^{(k)},\overline{\rho_{rs}^{(m)}}\rangle_{L^2(H)}=\frac{|H|}{d_{\pi_k}}\delta_{pr}\delta_{qs}, 
\]
which in particular implies that $\widehat{F_{il}}(\rho)\neq 0$. Summarizing, under our assumption that $\pi_{il}$ is not constantly zero over $H$, we have
\[
    \mathsf{supp}(\widehat{F_{il}})=\{\rho\in\widehat{\Gamma}:\ \pi_k\text{ appears in the decomposition in irred. components of }\rho|_H\}.
\]
\end{proof}

\begin{remark}
    Consider the abelian case $\Gamma=\mathbb{T}_{M,N}^d=\left(\frac{1}{M}\Z/N\Z\right)^d$, $X=(\Z/N\Z)^d$. By a dimensionality argument, every irreducible representation $e^{-2\pi i\langle\cdot,\lambda\rangle}$ of $\widehat{\mathbb{T}}_{M,N}^d\cong\mathbb{T}_{N,M}^d$ restricted to $X$ is in fact an irreducible representation of $\widehat{X}=X$. Thus, for $F_{\lambda}:=e^{-2\pi i\langle\cdot,\lambda\rangle}\delta_X$, from \Cref{suppFil} part 1 we get that 
    \[
    \mathsf{supp}(\widehat{F_{\lambda}})=\{\alpha\in\mathbb{T}_{N,M}^d:\ N\alpha+N\lambda\equiv 0(mod\ M)\}=\mathbb{T}_M^d-\lambda=X^*-\lambda,
    \]
    which coincides with the support of $\widehat{F_{\lambda}}$ given in \eqref{suppFlamb} below.
\end{remark}

\begin{proof}[Proof of point 2. of \Cref{suppFil}]
We now use point 1. of \Cref{suppFil} to compute $|\widehat{F_{il}}|_{L^2(\Gamma)}$. If $\rho\in\mathsf{supp}(\widehat{F_{il}})$, then by \Cref{Schurortho} we obtain a coordinate expression
\[
    \widehat{F_{il}}(\rho)=A_1\oplus A_2\oplus\ldots\oplus A_{l(\rho)},\qquad\text{where }A_m\in M(d_{\rho_m}\times d_{\rho_m},\mathbb{C})\text{ for every }1\leq m\leq l(\rho), 
\]
and where from \eqref{pirhol2} there holds that $A_m=0$ if $\rho_m\not\cong\pi_k$, whereas for $\rho_m\cong \pi_k$ we have
\[
    A_m=(A_{rs}^{(m)})_{r,s=1}^{d_{\pi_k}}\qquad\text{where }A_{rs}^{(m)}=\frac{|H|}{d_{\pi_k}|\Gamma|}\delta_{pr}\delta_{qs}.
\]
Then, by definition of $|\cdot|_{L^{2}(\Gamma)}$ and the previous characterization of $\widehat{F_{il}}$, we get 
\begin{equation*}
    \begin{split}
        \|\widehat{F_{il}}\|_{L^(\widehat{\Gamma})}^2&=\frac{1}{|\Gamma|}\sum_{\rho\in\mathsf{supp}(\widehat{F_{il}})}d_{\rho}\mathsf{Tr}(\widehat{F_{il}}(\rho)\widehat{F_{il}}(\rho)^*)\\
        &=\frac{1}{|\Gamma|}\sum_{\rho\in\mathsf{supp}(\widehat{F_{il}})}d_{\rho}\frac{|H|^2}{d_{\pi_k}^2|\Gamma|^2}\cdot m(\pi_k,\rho|_H)\\
        &=\frac{|H|^2}{d_{\pi_k}^2|\Gamma|^3}\sum_{\rho\in\widehat\Gamma}d_{\rho}\cdot m(\pi_{k(i,l)},\rho|_H).
    \end{split}
\end{equation*}
This implies the claimed equality \eqref{Fil}, concluding the proof.
\end{proof}
To estimate the value of $|\widehat{G_{lj}}|_{L^2(\Gamma)}$, we set $\widetilde{G_{lj}}(g):=\overline{G_{lj}(g^{-1})}$, and from the formula for the Fourier Transform of a convolution \eqref{nonabelconv}, the definition of $\widetilde{G_{lj}}$ and the property $\rho(g^{-1})=\rho(g)^*$ for every unitary representation $\rho$, we get
\begin{equation}
    \mathcal{F}(G_{lj}*\widetilde{G_{lj}})=|\Gamma|\widehat{G_{lj}}\widehat{G_{lj}}^*,
\end{equation}
and as we have $|A|^2=\mathsf{Tr}(AA^*)$ for the Frobenius norm of a matrix, we find
\begin{equation}\label{hatgsquare}
    \left|\widehat{G_{lj}}(\rho)\right|^2=\frac{|H|}{|\Gamma|}\mathsf{Tr}\left(\widehat{\gamma_{lj}}(\rho)\right), 
\end{equation}
where $\gamma_{lj}$ denotes the \emph{autocorrelation of $G_{lj}$} which is defined by
\begin{equation}\label{nonabelautocorr}
    \gamma_{lj}(g):=\frac{1}{|H|}\sum_{h\in H}g_{lj}(h)\overline{g_{lj}(g^{-1}h)}=\frac{1}{|H|}G_{lj}*\widetilde{G_{lj}}=\frac{1}{|H|}\sum_{h\in H}G_{lj}(h)\overline{G_{lj}(g^{-1}h)},
\end{equation}
where $\mathsf{supp}(\gamma_{lj})=H$. Hence, to estimate \eqref{nonabelest}, in view of \eqref{hatgsquare}, it is sufficient to find with high probability the approximate value of $\gamma_{lj}$.

\medskip

Since the $G_{lj}$ are random, it is natural to use Chebyshev inequality to estimate \eqref{nonabelautocorr}. Observe that the random variables $G_{lj}(h)\overline{G_{lj}(g^{-1}h)}$, with $h\in H$, are not necessarily independent. Hence, for each $g\in H$ the goal is decompose $H$ into three subsets $H_{g,i}$, $i=1,2,3$, of cardinalities close to $|H|/3$, for which the random variables $G_{lj}(h)\overline{G_{lj}(g^{-1}h)}$, with $h$ running over each of the $H_{g,i}$'s, are independent.
\begin{lemma}\label{parti3nonabel} 
For each $g\in H\setminus\{\mathsf{id}\}$ we can partition the subgroup $H$ into three subsets $H_{g,0}, H_{g,1}, H_{g,2}$ of cardinalities satisfying $||H_{g,i}| - |H|/3|\le 2$, so that for any $i=0,1,2$ and every $h,h'\in H_{g,i}$, with $h\neq h'$, there holds that $\{h,g^{-1}h\}\cap\{h',g^{-1}h'\}=\emptyset$.
\end{lemma}
\begin{proof}
Let us to start by listing the elements of $H$ by $\{\mathsf{id}=h_1,h_2,\ldots,h_{|H|}\}$ and define $h_{i_1}:=h_1$. For every $g\in H\setminus\{\mathsf{id}\}$, let $\mathcal{T}_g(h_{i_1})$ be the orbit of $h_{i_1}$ under left-translation of $g$, i.e:
\[
    \mathcal{T}_g(h_{i_1}):=\{g^{k}h_{i_1}:\ 0\leq k<\mathsf{ord}(g)\}.
\]
Then, let $h_{i_2}$ be the next element of the list such that $h_{i_2}\not\in\mathcal{T}_g(h_{i_1})$, and define the orbit
\[
    \mathcal{T}_g(h_{i_2}):=\{g^{k}h_{i_2}:\ 0\leq k<\mathsf{ord}(g)\}.
\]
Inductively and as long as $H\setminus\bigcup_{\alpha=1}^{m-1}\mathcal{T}_g(h_{i_{\alpha}})\neq\emptyset$, we consider $h_{i_{m}}\in H$ being the next element of the list which does not belong to some of the previous orbits and define:
\[
    \mathcal{T}_g(h_{i_m}):=\{g^kh_{i_m}:\ 0\leq k<\mathsf{ord}(g)\}.
\]
Let $m_H$ be the largest positive integer for which $H=\bigcup_{\alpha=1}^{m_H}\mathcal{T}_g(h_{i_{\alpha}})$. Since all the orbits $\mathcal{T}_g(h_{i_{\alpha}})$, $1\leq\alpha\leq m_H$ have the same number of points, the conclusions of \Cref{parti3nonabel} hold for the following subsets of $H$:
\[
    H_{g,i}:=\left\{g^{3k+i}h_{i_{\alpha}}:\ 0\leq k\leq\frac{\mathsf{ord}(g)}{3},\ 1\leq\alpha\leq m_H\right\},\qquad\text{where }i\in\{0,1,2\}.
\]
\end{proof}
\begin{prop}\label{nonabelvariance} 
Let $L>3$ and assume that $|H|\geq 6L/(L-3)$ and $\varepsilon>\left(\frac{12C^2L}{|H|}\right)^{1/2}$. Under our hypotheses on $\xi$, with probability at least $1-\frac{12C^2L}{\varepsilon^2|H|}$ there holds
\begin{equation}\label{nonabelprop4.4}
    \left|\gamma_{lj}(g)-\mathbb{E}\left[|G_{lj}(\mathsf{id})\overline{G_{lj}(g^{-1})}|^2\right]\delta_{\mathsf{id}}(g)\right|\leq \varepsilon.
\end{equation}
\end{prop}
\begin{proof}
For $g\neq\mathsf{id}$, let $H_{g,i}$ be as in \Cref{parti3nonabel}. By independence of the $\xi_g$, we have that the random variables $\{G_{lj}(h)\overline{G_{lj}(g^{-1}h)}, h\in H_{q,i}\}$ are i.i.d. By Chebyshev inequality, since $\mathsf{Var}[G_{lj}(h)\overline{G_{lj}(g^{-1}h)}]\le4C^2$, where $C=\sup_{\pi\in\widehat{\Gamma}}\mathsf{Var}[\pi(\xi)]$, we find for $i\in\{0,1,2\}$ and for $S_{g,i}:=\sum_{h\in H_{g,i}}G_{lj}(h)\overline{G_{lj}(g^{-1}h)}$:
\begin{equation}\label{cheby-simple-nonabel}
    \mathbb P\left(\left|\frac{1}{|H_{g,i}|}\sum_{h\in H_{g,i}}G_{lj}(h)\overline{G_{lj}(g^{-1}h)}-\mathbb E\left[G_{lj}(\mathsf{id})\overline{G_{lj}(g^{-1})}\right]\right|\geq\varepsilon\right)\le \frac{\mathsf{Var}(S_{g,i})}{\varepsilon^2|H_{g,i}|^2}\leq\frac{4C^2}{\varepsilon^2|H_{g,i}|}\le\frac{4C^2L}{\varepsilon^2|H|},
\end{equation}
where we used that $|H|\geq 6L/(L-3)$ and \Cref{parti3nonabel} to get that
\[
    \frac{1}{|H_{g,i}|}\leq\frac{L}{|H|}. 
\]
Denote by $\mathcal A_\varepsilon$ the event on the left in \eqref{cheby-simple-nonabel}. 
On this event, by triangle inequality, we find the following bound similar to the strong law of large numbers for the variables $G_{lj}(h)\overline{G_{lj}(g^{-1}h)}$:
\begin{equation}\label{nonabelcheby2}
    \mathbb P\left(\left|\gamma_{lj}(g) - \mathbb E\left[G_{lj}(\mathsf{id}) \overline{G_{lj}(g^{-1})}\right]\right|\geq\varepsilon\right)\le \frac{12C^2L}{\varepsilon^2|H|}.
\end{equation}
Since the random variables $G_{lj}(h)$, $h\in H$ are i.i.d with zero-mean, for every $g\neq\mathsf{id}$ we get that
\begin{equation}\label{qslln1}
    \mathbb{P}(|\gamma_{lj}(g)|\geq\varepsilon)\leq\frac{12C^2L}{\varepsilon^2|H|}. 
\end{equation}
On the other hand, for $g=\mathsf{id}$ and since the random variables $\{|G_{lj}(h)|^2,:\ h\in H\}$ are i.i.d, the very same calculation shows that:
\begin{equation}\label{qslln2}
    \mathbb{P}\left(\left|\frac{1}{|H|}\sum_{h\in H}|G_{lj}(h)|^2-\mathbb{E}(|G_{lj}(\mathsf{id})|^2)\right|\geq\varepsilon\right)\leq\frac{4C^2}{\varepsilon^2|H|}.
\end{equation}
Therefore, from \eqref{qslln1} and \eqref{qslln2} we get \eqref{nonabelprop4.4} as claimed.
\end{proof}
\begin{corollary}\label{cornonabelglj1}
For $L,\varepsilon$ and $\xi$ as in \Cref{nonabelvariance}, for each $\rho\in\widehat{\Gamma}$ there holds
\begin{equation}\label{estnormglj1}
    \left|\left|\widehat{G_{lj}}(\rho)\right|^2-\frac{|H|}{|\Gamma|^2}\cdot\mathbb E\left[|G_{lj}(\mathsf{id})|^2\right]\ d_\rho\right|\le \varepsilon d_\rho\frac{|H|^2}{|\Gamma|^2}.
\end{equation}
\end{corollary}
\begin{proof}
We concentrate on the term on the left-hand side in \eqref{hatgsquare}, and we find, by triangle inequality and using the fact that $\max_{h\in H}|\mathsf{Tr}(\rho(h))|=d_\rho$ and from \Cref{nonabelvariance}, we get
\begin{equation*}
    \begin{split}
        &\left|\mathsf{Tr}\left(\widehat{\gamma_{lj}}(\rho)- \mathcal{F}\left(\mathbb E\left[G_{lj}(\mathsf{id})\overline{G_{lj}((\cdot)^{-1})}\right]\delta_{\mathsf{id}}\right)(\rho)\right)\right|\\
        &\leq\frac{1}{|\Gamma|}\sum_{g\in \Gamma}\left|\gamma_{lj}(g) - \mathbb E\left[G_{lj}(\mathsf{id})\overline{G_{lj}(g^{-1})}|\right]\delta_{\mathsf{id}}(g)\right|\left|\mathsf{Tr}(\rho(g))\right|\leq \varepsilon d_\rho\frac{|H|}{|\Gamma|},
    \end{split}
\end{equation*}
where for the second inequality, we also use the fact that $\mathsf{supp}(\gamma_{lj})=H$. Now note that

\[
\mathcal{F}\left(\mathbb E\left[G_{lj}(\mathsf{id})\overline{G_{lj}((\cdot)^{-1})}\right]\delta_{\mathsf{id}}\right)(\rho)=\frac{1}{|\Gamma|}\sum_{g\in\Gamma}\mathbb E\left[G_{lj}(\mathsf{id})\overline{G_{lj}(g^{-1})}\right]\delta_{\mathsf{id}}(g)\rho(g)=\mathbb{E}\left[|G_{lj}(\mathsf{id})|^2\right]\cdot\frac{\mathsf{Id}_{d_{\rho}}}{|\Gamma|},
\]

and then $\mathsf{Tr}\left(\mathcal{F}\left(\mathbb E\left[G_{lj}(\mathsf{id})\overline{G_{lj}((\cdot)^{-1})}\right]\delta_{\mathsf{id}}\right)(\rho)\right)=\mathbb{E}\left[|G_{lj}(\mathsf{id})|^2\right]\frac{d_\rho}{|\Gamma|}$. Together with \eqref{hatgsquare} this gives
\[
    \left|\frac{|\Gamma|}{|H|}\left|\widehat{G_{lj}}(\rho)\right|^2- \mathbb E\left[|G_{lj}(\mathsf{id})|^2\right]\frac{d_\rho}{|\Gamma|}\right|\le \varepsilon d_\rho\frac{|H|}{|\Gamma|}.
\]
the desired bound \eqref{estnormglj1} follows directly.
\end{proof}
\begin{corollary}\label{coronabelglj2}
Under the hypotheses of \Cref{cornonabelglj1}, there holds
\begin{equation}\label{estnormglj2}
   \left|\left\|\widehat{G_{lj}}\right\|_{L^2(\widehat{\Gamma})}^2-\frac{|H|}{|\Gamma|^2}\cdot\mathbb{E}\left[|G_{lj}(\mathsf{id})|^2\right]\right|\leq \frac{\varepsilon|H|^2}{|\Gamma|^2}.
\end{equation}
\end{corollary}
\begin{proof}
By the triangle inequality, we find
\[
    \left|\left\|\widehat{G_{lj}}\right\|^2_{L^2(\widehat\Gamma)} - \frac{|H|}{|\Gamma|^2}\cdot\mathbb E\left[|G_{lj}(\xi)|^2\right]\right|\le \frac{1}{|\Gamma|}\sum_{\rho\in\widehat{\Gamma}}d_\rho\left|\left|\widehat{G_{lj}}(\rho)\right|^2 - \frac{|H|}{|\Gamma|^2}\cdot\mathbb E\left[|G_{lj}(\xi)|^2d_\rho\right]\right|,
\]
and then \eqref{estnormglj2} follows by applying \eqref{estnormglj1} and \eqref{dimrep}.
\end{proof}
\begin{proof}[Proof of \Cref{basicthmH}]
From \Cref{suppFil}, \Cref{coronabelglj2} and \eqref{nonabelest} we can bound \eqref{estFT} from above by the below (in which $\pi_{k(i,l)}$ is as in \Cref{suppFil} and $C:=\sup_{\pi\in\widehat{\Gamma}}\mathsf{Var}[\pi(\xi)]$):
\begin{equation}\label{upperbound1}
    \begin{split}
        &|\Gamma|\sum_{\pi\in\widehat{\Gamma}}d_{\pi}\sum_{i,j=1}^{d_{\pi}}\left(\sum_{l=1}^{d_{\pi}}\frac{\mathbbm{1}_{\pi_{il}(H)\neq\{0\}}}{d_{\pi_{k(i,l)}}}\left\{\frac{|H|^2}{|\Gamma|^3}\sum_{\rho\in\widehat\Gamma}d_{\rho}\cdot m(\pi_{k(i,l)},\rho|_H)\right\}^{1/2}\left\{\frac{\varepsilon|H|^2}{|\Gamma|^2}+\frac{|H|}{|\Gamma|^2}\mathbb{E}\left[|G_{lj}(\mathsf{id})|^2\right]\right\}^{1/2}\right)^2\\
        &\leq\frac{|H|^2}{|\Gamma|^2}\left(\frac{\varepsilon|H|^2}{|\Gamma|^2}+\frac{4C^2|H|}{|\Gamma|^2}\right)\sum_{\pi\in\widehat{\Gamma}}d_{\pi}^2\sum_{i=1}^{d_{\pi}}\left(\sum_{l=1}^{d_{\pi}}\frac{\mathbbm{1}_{\pi_{il}(H)\neq\{0\}}}{d_{\pi_{k(i,l)}}}\left\{\sum_{\rho\in\widehat\Gamma}d_{\rho}\cdot m(\pi_{k(i,l)},\rho|_H)\right\}^{1/2}\right)^2\\
        &=\frac{|H|^4}{|\Gamma|^4}\left(\varepsilon+\frac{4C^2}{|H|}\right)\sum_{\pi\in\widehat{\Gamma}}\frac{d_{\pi}^2}{d_{\pi_1}^2}\sum_{i=1}^{d_{\pi}}\left(\sum_{l=1}^{d_{\pi}}\mathbbm{1}_{\pi_{il}(H)\neq\{0\}}\left(\sum_{\rho\in\widehat\Gamma}d_{\rho}\cdot m(\pi_1,\rho|_H)\right)^{1/2}\right)^2,
    \end{split}    
\end{equation}
where in the last step, we used the fact that since $H$ is a normal subgroup, by \Cref{cliffordthm}, the dimensions $d_{\pi_k}$ are independent of $k$ and equal, e.g., to $d_{\pi_1}$.

\medskip

Recall that $\pi_{il}$ are the matrix coefficients of $\pi$ in a basis in which $\pi|_H$ has block structure with $d_{\pi_1}\times d_{\pi_1}$-dimensional blocks, thus for each value of index $i$, at most $d_{\pi_1}$ of the indices $i,l$ can give nonzero contributions. Furthermore, $m(\pi_{k(i,l)},\rho|_H)\le d_\rho$. Therefore we have 
\begin{equation}\label{innersumest}
    \sum_{l=1}^{d_{\pi}}1_{\pi_{il}(H)\neq\{0\}}\left(\sum_{\rho\in\widehat\Gamma}d_{\rho}\cdot m(\pi_1,\rho|_H)\right)^{1/2}\leq \sum_{l=1}^{d_{\pi}}\mathbbm{1}_{\pi_{il}(H)\neq\{0\}}\left(\sum_{\rho\in\widehat\Gamma}d_{\rho}^2\right)^{1/2}\leq |\Gamma|^{1/2}d_{\pi_1}.
\end{equation}
Inserting \eqref{innersumest} into \eqref{upperbound1} we get
\begin{equation*}\label{upperbound2}
    \left\|\widehat{\delta_{H_{\xi}}}-\widehat{\delta_H}\cdot\mathbb{E}(\xi(\cdot))\right\|_{L^2(\widehat{\Gamma})}^2\leq \frac{|H|^4}{|\Gamma|^3}\left(\varepsilon+\frac{4C^2}{|H|}\right)\sum_{\pi\in\widehat{\Gamma}}d_\pi^3\leq\frac{|H|^4}{|\Gamma|^2}\left(\varepsilon+\frac{4C^2}{|H|}\right)\max_{\pi\in\widehat\Gamma}d_\pi.
\end{equation*}
Thus we obtain the desired bound, concluding the proof of \Cref{basicthmH}.
\end{proof}

\begin{remark}\label{rmk: basicthmHsubset}
    For $A\subset\widehat{\Gamma}$, the very same proof of \Cref{basicthmH} shows that
    \begin{equation*}\label{basicthmHsubset}
        \begin{split}
            \left\|\widehat{\delta_{H_{\xi}}}-\widehat{\delta_H}\cdot\mathbb{E}(\xi(\cdot))\right\|_{L^2(A)}^2&\leq \frac{|H|^4}{|\Gamma|^3}\left(\varepsilon+\frac{4C^2}{|H|}\right)\sum_{\pi\in A}d_{\pi}^3\\
            &\leq\frac{|H|^4}{|\Gamma|^2}\left(\varepsilon+\frac{4C^2}{|H|}\right)\max_{\pi\in A}d_{\pi},
        \end{split}
    \end{equation*}
where $\|F\|_{L^2(A)}^2:=\frac{1}{|\Gamma|}\sum_{\pi\in A}d_{\pi}\mathsf{Tr}(F(\pi)F(\pi)^*)$ for $F\in L^2(\widehat{\Gamma})$. In particular, for $A=\{\pi\}$, with $\pi\in\widehat{\Gamma}$ and for (any) orthonormal basis $(e_i^{(\pi)})_{i=1}^{d_{\pi}}$ it follows that
\[
     \left\|\widehat{\delta_{H_{\xi}}}-\widehat{\delta_H}\cdot\mathbb{E}(\xi(\cdot))\right\|_{L^2(\{\pi\})}^2=\frac{d_{\pi}}{|\Gamma|}\sum_{i=1}^{d_{\pi}} \left\|\left(\widehat{\delta_{H_{\xi}}}(\pi)-\widehat{\delta_H}(\pi)\cdot\mathbb{E}(\pi(\xi))\right)e_i^{(\pi)}\right\|_{V_{\pi}}^2\leq \frac{|H|^4}{|\Gamma|^3}\left(\varepsilon+\frac{4C^2}{|H|}\right)d_{\pi}^3,
\]
and thus we get the following estimation
\begin{equation}\label{basicthmHpw}
    \sum_{i=1}^{d_{\pi}} \left\|\left(\widehat{\delta_{H_{\xi}}}(\pi)-\widehat{\delta_H}(\pi)\cdot\mathbb{E}(\pi(\xi))\right)e_i^{(\pi)}\right\|_{V_{\pi}}^2\leq\frac{|H|^4}{|\Gamma|^2}\left(\varepsilon+\frac{4C^2}{|H|}\right)d_{\pi}^2. 
\end{equation}
\end{remark}

%%%%%%%%%%%%%%%%%%%%%%%%%%%%%%%%%%%%%%%%%%%%%%%%%%%%%%%%%%%%%%%%%%%%%%%%%%%%%%%%%%%%%%%%%%%%%%%%%%%%%%%%%%%%%%%%%%%%%%%%%%%%%%%%%%%%%%%%
\section{Some elements on Fourier Analysis in nilpotent Lie groups}\label{sec: recovgennil}
\subsection{Uniform subgroups in nilpotent Lie groups}\label{sec: lattnil}

We start this section with some basic background on nilpotent Lie groups; for further details, see for instance \cite{corgreen}. Let $\Gamma$ be a $d$-dimensional locally compact simply connected (lcsc) Lie group, and denote by $\mathfrak g$ its Lie algebra, with associated Lie bracket $[\cdot,\cdot]:\mathfrak g\times\mathfrak g\to\mathfrak g$. It is well-known that the exponential map $\mathsf{exp}:\mathfrak g\to \Gamma$ produces a global diffeomorphism between $\Gamma$ and its Lie algebra $\mathfrak g\cong\R^d$, and thus we can think $\Gamma$ topologically as $\R^d$. Moreover, the exponential map is a Lie isomorphism when $\mathfrak g$ is endowed with the group law given by the Baker-Campbell-Hausdorff formula.
\medskip

For the Lie algebra $\mathfrak{g}$ we define its {\em lower central series}
as the sequence of Lie subalgebras $\mathfrak{g}\trianglerighteq\mathfrak{g}_1\trianglerighteq\mathfrak{g}_2\trianglerighteq\ldots$, where 

\[\mathfrak{g}_0:=\mathfrak{g}\qquad\text{and}\qquad \mathfrak{g}_i:=[\mathfrak{g},\mathfrak{g}_{i-1}].\]
We say that $\Gamma$ is {\em nilpotent} if there exists a positive integer $s$ such that $\mathfrak{g}_{s+1}=\{0\}$; when $\mathfrak{g}_s\neq\{0\}$ and $\mathfrak{g}_{s+1}=\{0\}$, then $G$ is {\em nilpotent of step $s$}.
\medskip 

Now, assume that $\Gamma$ is nilpotent of step $s$. A subset $\{X_1,\ldots,X_d\}\subset\Gamma$ is called a {\em strong Mal'cev basis} subordinated to $\{0\}=\mathfrak{g}_{s+1}\subset\mathfrak{g}_s\subset\ldots\subset\mathfrak{g}_2\subset\mathfrak{g}_1\subset\mathfrak{g}_0=\mathfrak{g}$ if for each $0\leq i\leq s$ there holds that 
\begin{equation}\label{malcevbasis}
    \mathfrak{g}_i=\mathsf{span}\{X_1,\ldots,X_{d_i}\},\qquad\text{where}\qquad d_i:=\mathsf{dim}(\mathfrak{g}_i). 
\end{equation}

For $X_j,X_k$ in the Mal'cev basis, $j\leq k$, there exists a natural number $i$ such that $X_k\in\mathfrak{g}_i$, and in particular we get that $[X_j,X_k]\in\mathfrak{g}_{i+1}$. Since $\{X_1,\ldots,X_d\}$ is a strong Mal'cev basis, there are constants $c_{jk}^l$, $1\leq l\leq d_{i+1}$ (called {\em structural constants}), such that 

\[
[X_j,X_k]=\sum_{l=1}^{d_{i+1}}c_{jk}^{l}X_l.
\]

For a Mal'cev basis as above, we can establish a diffeomorphism $\phi:\R^d\to\Gamma$ given by the formula
\begin{equation}\label{diffnil}
\phi(x_1,\ldots,x_d):=\mathsf{exp}(x_1X_1)\cdot\mathsf{exp}(x_2X_2)\cdot\ldots\cdot\mathsf{exp}(x_dX_d),
\end{equation}

which by the Baker-Campbell-Hausdorff formula (see \cite[Proposition 1.2.7]{corgreen}), is a polynomial diffeomorphism of the form

\begin{equation}\label{polcoord}
    \phi(x)=\mathsf{exp}\left(\sum_{j=1}^d P_j(x)X_j\right),
\end{equation}

where the $P_j$'s are polynomials without constant coefficients, and 
\[P_j(x)=x_j+\text{ a polynomial on }x_{j+1},\ldots,x_d,\qquad\text{for }j=1,\ldots,d.\]

Observe that if $\Gamma$ admits a Mal'cev basis with rational structural constants, we can choose the polynomials $P_j$ with integer coefficients.
\medskip

If $\Gamma$ is a lcsc nilpotent Lie group, by a {\em uniform subgroup} in $\Gamma$, we mean a discrete subgroup $H\subset\Gamma$ such that $\Gamma/H$ is compact. The nature of the structural constants is intrinsically related to the existence of a uniform subgroup in $\Gamma$. Indeed, the way to connect both concepts is given by the next (see \cite[Theorem 5.1.8]{corgreen}):
\begin{equation*}
\begin{split}
\Gamma\text{ admits a uniform subgroup }H\Longleftrightarrow\mathfrak{g}\text{ admits a Mal'cev basis with rational structural constants.} 
\end{split}
\end{equation*}

Let $H\subset\Gamma$ be uniform and let $K\subset\Gamma$ be a subgroup. In this case, it is known that $K$ is rational (i.e., its Lie algebra admits a Mal'cev basis with rational structural constants) if and only if $H\cap K$ is uniform in $K$ (see \cite[Theorem 5.1.11]{corgreen}); this fact will be exploited in \Cref{ssec: nildisc}.					

%%%%%%%%%%%%%%%%%%%%%%%%%%%%%%%%%%%%%%%%%%%%%%%%%%%%%%%%%%%%%%%%%%%%%%%%%%%%%%%%%%%%%%%%%%%%%%%%%%%%%%%%%%%%%%%%%%%%%%%%%%%%%%%%%%%%%%%%%%%%%%%%
\subsection{Induced representations}\label{ssec: indrep}

In this section, we give the basic and necessary elements of induced representations that will be used for the problem of recovering the discrete Fourier spectrum in nilpotent groups; for more details, see \cite[Chapter 6]{folland}. Induced representations are representations of a group $\Gamma$ constructed from a smaller subgroup $K$. Let $\Gamma$ be a locally compact group, $K$ be a closed subgroup, and $i:\Gamma\to\Gamma/K$ be the canonical quotient map that sends an element $x\in\Gamma$ to its equivalence class $[x]$, and where $\Gamma/K$ denotes the space of right-cosets. Consider a unitary representation $\sigma: K\to \mathcal{U}(\mathcal{H}_{\sigma})$, where the inner product and the norm in the Hilbert space $\mathcal{H}_{\sigma}$ are denoted by $\langle\cdot,\cdot\rangle_{\mathcal{H}_{\sigma}}$ and $||\cdot||_{\mathcal{H}_{\sigma}}$, respectively. 
\medskip

For our purposes, assume that $\Gamma/K$ admits a left-invariant measure $\eta$ (this is the case, for instance, when $K$ is a closed subgroup of a nilpotent Lie group $\Gamma$). Let $\mathcal{F}_{\sigma}$ be the space of all the Borel measurable functions $f:\Gamma\to\mathcal{H}_{\sigma}$ such that 

\begin{equation}\label{Fsigma}
    \int_{\Gamma/K}\|f(x)\|^2d\eta(xK)<\infty\qquad\text{and}\qquad f(\xi x)=\sigma(\xi)f(x)\text{ for }x\in\Gamma, \xi\in K.
\end{equation}

Observe that for $f,g\in\mathcal{F}_{\sigma}$ and $x\in\Gamma$, the inner product $\langle f(x),g(x)\rangle_{\sigma}$ only depends on the coset $i(x)$. Then, we can define an inner product in $\mathcal{F}_{\sigma}$ by means of the formula:

 \begin{equation}\label{innerprodF}
    \langle f,g\rangle=\langle f,g\rangle_{\mathcal{F}_{\sigma}}:=\int_{\Gamma/K} \langle f(x),g(x)\rangle_{\mathcal{H}_{\sigma}} d\eta(xK).
 \end{equation}
 
 Thus, the inner product \eqref{innerprodF} becomes $\mathcal{F}_{\sigma}$ in a complete Hilbert space. Note that the left-translation $x\mapsto L_x:\mathcal{F}_{\sigma}\to\mathcal{F}_{\sigma}$, where $L_x(f)(y):=f(xy)$, is a unitary representation of $\Gamma$. Hence, we define the {\em induced representation by $\sigma$} as the unitary representation $\mathsf{ind}_K^{\Gamma}(\sigma):\Gamma\to\mathcal{U}(\mathcal{F}_{\sigma})$ given by left-translations; more precisely

 \begin{equation}\label{indrep}
     [\mathsf{ind}_H^G(\sigma)(x)f](y):=f(xy).
 \end{equation}

For explicit computations of the Fourier transform, it is more useful to describe $\mathsf{ind}_K^{\Gamma}(\sigma)$ by acting over a suitable space of square-integrable functions. To achieve this, let $\alpha:\Gamma/K\to\Gamma$ be a Borel cross-section of $i:\Gamma\to\Gamma/K$, i.e., for which there holds $i\circ\alpha=\mathsf{id}_{\Gamma/K}$. Due to \eqref{Fsigma} we have that $f\in\mathcal{F}_{\sigma}$ is completely determined by its values on $\alpha (\Gamma/K)$; indeed, we have that for every $x\in\Gamma$ there exist $h=h(x)\in K$ and $k=k(x)\in\alpha(\Gamma/K)$ such that $x=h(x)k(x)$, and then
\begin{equation}\label{decomcross}
    f(x)=f(hk)=\sigma(h)f(k).
\end{equation}
In addition, we have that the map $\mathcal{F}_{\sigma}\ni f\mapsto f\circ\alpha\in L^2(\Gamma/K)$ becomes in an isometry.

%%%%%%%%%%%%%%%%%%%%%%%%%%%%%%%%%%%%%%%%%%%%%%%%%%%%%%%%%%%%%%%%%%%
%%%%%%%%%%%%%%%%%%%%%%%%%%%%%%%%%%%%%%%%%%%%%%%%%%%%%%%%%%%%%%%%%%%%%%%%%%%%%%%%%%%%%%%%%%%%%%%%%%%%%%%%%%%%%%%%%%%%%%%%%%%%%%%%%%%%%%%%
\subsection{Elements of Kirillov's orbit method}\label{ssec: Kirillov}

In this section, we introduce the necessary ingredients of the well-understood theory of irreducible representations in nilpotent groups developed in the remarkable work of Kirillov \cite{kir}. Henceforth let $\Gamma$ be an lcsc nilpotent Lie group. For $g\in \Gamma$ denote by $\Psi_x:\Gamma\to\Gamma$ the conjugation given by $\Psi_g(x)=gxg^{-1}$ and consider the map $\mathsf{Ad}_g:=(\Psi_g)_*:\mathfrak g\to \mathfrak g$, where

\[
((\Psi_g)_* X)_p:=(d\Psi_g)_{\Psi_g^{-1}(p)}X_{\Psi_g^{-1}(p)},\qquad\text{for }X\in\mathfrak g.
\]

The {\em adjoint representation} corresponds to the map $\mathsf{Ad}:\Gamma\to GL(\mathfrak g)$ given by $\mathsf{Ad}(g):=\mathsf{Ad}_g$. The {\em coadjoint map} $\mathsf{Ad}^*:\Gamma\to GL(\mathfrak g^*)$ is defined by the formula

\begin{equation}\label{coadjoint}
    [(\mathsf{Ad}^*x)l](Y):=l(\mathsf{Ad}x^{-1}Y),\qquad\text{where }x\in\Gamma, l\in\mathfrak g^*,\text{and }Y\in\mathfrak g.
\end{equation}

Formula \eqref{coadjoint} defines an action of $\Gamma$ over $\mathfrak g^*$ given by $x\cdot l:=(\mathsf{Ad}^* x)l$. %Denote by $R_l$ be the stabilizer of $l\in\mathfrak g^*$ under the coadjoint action, and   
Each $l\in\mathfrak g^*$ determines a symplectic form $B_l:\mathfrak g\times\mathfrak g\to\mathfrak g$ defined as follows:
\begin{equation}\label{bilinear}
B_l(X,Y):=l([X,Y]),\quad\text{where }X,Y\in\mathfrak g.
\end{equation}

The {\em radical} of $B_l$ is the set $r_l:=\{Y\in\mathfrak g:\ B_l(X,Y)=0\text{ for every }X\in\mathfrak g\}$. For lcsc nilpotent Lie groups, there is a relation between the radical of $l\in\mathfrak{g}^*$ and the stabilizer $\mathsf{Stab}_l:=\{g\in\Gamma: g\cdot l=l\}$ under the coadjoint action, namely, that 
\[
\mathsf{Stab}_l=\mathsf{exp}(r_l).
\]

An {\em isotropic} subspace of $\mathfrak g$ for $B_l$ is a subspace $W\subset\mathfrak g$ for which $B_l(X,X')=0$ for every $X,X'\in W$. There are maximal subalgebras $\mathfrak m\subset\mathfrak g$ that are isotropic for $B_l$ and have (maximal) isotropic dimension $d-k$, where 
\[
k=\frac{1}{2}\mathsf{dim}(\mathfrak{g}/r_l).
\] 
Such subalgebras are called {\em polarizing subalgebras} for $l$, which always exist (see \cite[Theorem 1.3.3]{corgreen}).
\medskip

Let $\mathfrak m=\mathfrak m_l$ be a polarizing subalgebra for a fixed $l\in\mathfrak{g}^*$, and denote $K_l:=\mathsf{exp}(\mathfrak m)$. By the isotropy, the map $\chi_{l,\mathfrak m}:K_l\to S^1$ given by
\begin{equation}
    \chi_{l,\mathfrak m}(\mathsf{exp}(Y)):=e^{2\pi i l(Y)},\quad\text{where }Y\in\mathfrak m,
\end{equation}
is a $1$-dimensional irreducible representation of $K_l$. Define $\pi_{l,\mathfrak m}:=\mathsf{ind}_K^{\Gamma}(\chi_{l,\mathfrak m}):\Gamma\to\mathcal{U}(\mathcal{H}_l)$, where $\mathcal{H}_l$ is the Hilbert space where $\pi_l$ acts, defined like in \Cref{ssec: indrep}. The next theorem is the cornerstone of Kirillov's theory and claims that the above-inducing mechanism produces {\em all} the irreducible representations of the nilpotent group $\Gamma$. 

\begin{thm}\label{kirivmeth}
    Let $l\in \mathfrak g^*$. There exists a polarizing subalgebra $\mathfrak m=\mathfrak{m}_l$ for $l$ such that $\pi_{l,\mathfrak m}$ is irreducible. Moreover, every irreducible representation of $\Gamma$ arises in this way. Also, the following hold:
    \begin{enumerate}
        \item if $\mathfrak m, \mathfrak{m}'$ are two polarizing subalgebras for $l$, then $\pi_{l,\mathfrak m}\cong\pi_{l,\mathfrak{m}'}$. Thus, we can write $\pi_{l}=\pi_{l,\mathfrak m}$.

        \item $\pi_l\cong\pi_{l'}$ if and only if $l$ and $l'$ belong to the same $\mathsf{Ad}^*(\Gamma)$-orbit in $\mathfrak g^*$.
    \end{enumerate}
\end{thm}

Following the discussion at the end of \Cref{ssec: indrep}, we point out that in the case of lcsc nilpotent Lie groups, for
\begin{equation}\label{dimindrep}
    s:=\mathsf{dim}(\Gamma/K_l)=d-\mathsf{dim}(\mathfrak{m}_l)=\frac{1}{2}\mathsf{dim}(\mathfrak{g}/r_l),
\end{equation}
we have that the Hilbert space $\mathcal{H}_l$ is isometric to $L^2(\R^s)$ under the mapping $\mathcal{H}_l\ni f\mapsto f\circ \alpha\in L^2(\Gamma/K_l)$, and where $\Gamma/K_l\cong\R^s$. In addition, we get that $s$ cannot exceed $d/2$.

\begin{lemma}\label{lemma:dimirrednil}
    Under the previous notations, we have that $s<d/2$.
\end{lemma}

\begin{proof}
    
Otherwise, from \eqref{dimindrep} we get that  
\[
\frac{d}{2}\leq s=\frac{d}{2}-\frac{1}{2}\mathsf{dim}(r_l)\Longrightarrow \mathsf{dim}(r_l)=0.
\]

This implies that $\mathsf{Stab}_l=\mathsf{exp}(r_l)=\{\mathsf{id}\}$. In particular, $\mathfrak{g}^*/\mathsf{Ad}^*\Gamma$ consists of only one orbit, and by \Cref{kirivmeth} we have that $|\widehat{\Gamma}|=1$. Then the trivial representation $\Gamma\ni g\mapsto 1\in\R/\Z$ is the only irreducible representation (under equivalence) of $\Gamma$, which is one dimensional and hence $\Gamma$ is Abelian, and by Pontryagin duality, we have that $\Gamma=\{\mathsf{id}\}$, which is impossible since $\Gamma\cong (\R^d,\cdot)$, where $d\geq 1$. 
\end{proof}
%%%%%%%%%%%%%%%%%%%%%%%%%%%%%%%%%%%%%%%%%%%%%%%%%%%%%%%%%%%%%%%%%%%%%%%%%%%%%%%%%%%%%%%%%%%%%%%%%%%%%%%%%%%%%%%%%%%%%%%%%%%%%%%%%%%%%%%%%%%%
\subsection{Basics on the Fourier transform in nilpotent Lie groups}\label{sec: FTnil}

In this section, we provide the basic elements of Fourier Analysis which we shall use to deal with the recovery problem. Although this theory can be presented with all its abstraction in second countably, unimodular Type I groups (see \cite{algebras, folland}), as was explained above (see \Cref{prop:bft}) we are concerned on lcsc nilpotent Lie groups, and thus we present the theory in this last context; we refer to \cite{quannil} for further details. Like in the Euclidean case, the basic space to define a Fourier transform in $\Gamma$ is the Schwartz space of rapidly decreasing functions. We say that a function $f:\Gamma\to\mathbb{C}$ is in the {\em Schwartz space} $\mathcal{S}(\Gamma)$ if $f\circ\mathsf{exp}\in\mathcal{S}(\mathfrak{g})$. For $f\in\mathcal{S}(\Gamma)$ we define its {\em group Fourier transform} by the linear mapping

\begin{equation}\label{groupFT}
    (\forall\pi\in\widehat{\Gamma}),\qquad\mathcal{F}(f)(\pi)=\widehat{f}(\pi):=\int_{\Gamma} f(x)\pi(x)dx:\mathcal{H}_{\pi}\to\mathcal{H}_{\pi};
\end{equation}
more precisely, for $\phi,\psi\in\mathcal{H}_{\pi}$ we have

\[
\langle\widehat{f}(\pi)\phi,\psi\rangle_{\mathcal{H}_{\pi}}=\int_{\Gamma} f(x)\langle\pi(x)\phi,\psi\rangle_{\mathcal{H}_{\pi}} dx.
\]
\medskip

We say that a positive operator $T$ on a Hilbert space $\mathcal{H}$ is {\em trace-class} if $T$ has an orthonormal basis $(e_n)_{n\geq 1}$ with eigenvalues $\lambda_n>0$ such that $\sum_{n\geq 1}\lambda_n<\infty$; in this case we define $\mathsf{Tr}(T):=\sum_{n\geq 1}\lambda_n$. In particular, if $T$ is a positive operator and trace-class, then for every orthonormal basis $(x_n)_{n\geq 1}$ for $\mathcal{H}$ we get $\mathsf{Tr}(T)=\sum_{n\geq 1}\langle Tx_n,x_n\rangle_{\mathcal{H}}$. Finally, we say that an operator $T\in\mathcal{B}(\mathcal{H})$ is {\em trace-class} if the positive operator $|T|:=\sqrt{TT^*}$ is trace-class, where the root of a positive operator $\sqrt{\cdot}$ is in the sense of spectral calculus.
\medskip 

In analogy to the classical Fourier transform in $\R^d$ as an isometric isomorphism $\mathcal{S}(\R^d)\to\mathcal{S}(\R^d)$ (and then extended to an operator from $L^2(\R^d)$ to itself), the nature of the space $\mathcal{F}(\mathcal{S}(\Gamma))$ is well-understood as a function space after $\widehat{\Gamma}$ is equipped with an appropriate measure, called the Plancherel measure below. 
\medskip

\begin{thm}[Thm. 1.8.2 in \cite{quannil}]\label{planmeas} There exists a $\sigma$-finite measure $\mu$ in $\widehat{\Gamma}$ called the {\em Plancherel measure} such that for every $\phi\in\mathcal{S}(\Gamma)$, the operator $\mathcal{F}(\phi)=\widehat{\phi}(\pi)\in\mathcal{B}(\mathcal{H}_{\pi})$ is trace-class for any strongly continuous unitary representation $\pi$, and $\mathsf{Tr}(\widehat{\phi}(\pi))$ depends only in the class of $\pi$. Moreover, the function $\pi\mapsto\mathsf{Tr}(\widehat{\phi}(\pi))$ is integrable against $\mu$ and there holds:
    \[\phi(0)=\int_{\widehat{\Gamma}}\mathsf{Tr}(\widehat{\phi}(\pi))d\mu(\pi).\]
\end{thm}

Henceforth we fix $\mu$ the Plancherel measure of $\widehat{\Gamma}$. Note that by applying \Cref{planmeas} to $\phi_x(y):=f(yx)$ and since $\widehat{\phi}(\pi)=\pi(x)\widehat{f}(\pi)$, we obtain the Fourier inversion formula

\[
f(x)=\int_{\widehat{\Gamma}}\mathsf{Tr}(\pi(x)\widehat{f}(\pi))d\mu(\pi)=\int_{\widehat{\Gamma}}\mathsf{Tr}(\widehat{f}(\pi)\pi(x))d\mu(\pi).
\]

An operator $T\in\mathcal{B}(\mathcal{H})$ is called {\em Hilbert-Schmidt} if $TT^*$ is trace-class; we write 

\[
\|T\|_{\mathsf{HS}(\mathcal{H})}^2:=\mathsf{Tr}(TT^*).
\]

The Fourier inversion formula applied to $\phi*\phi^*$, where $\phi^*(x)=\overline{\phi(x^{-1})}$ yields the Plancherel formula below.

\begin{thm}[Thm. 1.8.5 in \cite{quannil}]
Under the notation of \Cref{planmeas}, for $\phi\in\mathcal{S}(\Gamma)$, for every strongly continuous $\pi\in\widehat{\Gamma}$, the operator $\widehat{\phi}(\pi)$ is Hilbert-Schmidt, and $\|\widehat{\phi}(\pi)\|_{\mathsf{HS}}$ is constant in the equivalence class of $\pi$. Moreover $\pi\mapsto\|\widehat{\phi}(\pi)\|_{\mathsf{HS}(\mathcal{H}_{\pi})}^2$ is $\mu$-integrable and
    \[
    \|\phi\|_{L^2(\Gamma)}^2=\int_{\Gamma} |\phi(x)|^2dx=\int_{\widehat{\Gamma}}\|\widehat{\phi}(\pi)\|_{\mathsf{HS}(\mathcal{H}_{\pi})}^2 d\mu(\pi).
    \]
\end{thm}

Observe that $\widehat{\phi}(\pi)$ is a bounded operator on $\mathcal{H}_{\pi}$ for $\phi\in C_c(\Gamma)$. The Plancherel formula provides a first approach to the nature of the group Fourier transform as an operator, namely, $\mathcal{F}:C_c(\Gamma)\to L^2(\widehat{\Gamma})$ is an isometry, where the {\em direct integral}

\[
L^2(\widehat{\Gamma}):=\int_{\widehat{\Gamma}}^{\oplus}\mathsf{HS}(\mathcal{H}_{\pi})d\mu(\pi)
\]
is the Hilbert space of the $\mu$-measurable fields of Hilbert-Schmidt operators $A:\widehat{\Gamma}\ni\pi\mapsto A_{\pi}\in\mathsf{HS}(\mathcal{H}_{\pi})$ for which

\begin{equation}
    \|A\|_{L^2(\widehat{\Gamma})}^2:=\int_{\widehat{\Gamma}}\|A_{\pi}\|_{\mathsf{HS}(\mathcal{H}_{\pi})}^2 d\mu(\pi)<\infty.
\end{equation}

Here measurability means that 
\[
\widehat{\Gamma}\ni\pi\mapsto\langle A_{\pi}\varphi,\psi\rangle_{\mathcal{H}_{\pi}}\qquad\text{is measurable }\forall\varphi,\psi\in\mathcal{H}_{\pi};
\]
for additional details on direct integrals and operators fields, see 
\cite[Appendix B]{quannil} or \cite[Section 7.4]{folland}. Like in the Euclidean case, the group Fourier transform can be extended in an isometry $\mathcal{F}:L^2(\Gamma)\to L^2(\widehat{\Gamma})$ which is surjective. 
\medskip

Now, in order to compute the Fourier transform of the Dirac delta $\delta_p,p\in\Gamma$, we need to define the group Fourier transform of a measure. First, let $L^{\infty}(\widehat{\Gamma})$ be the space of $\mu$-measurable fields $A:\widehat{\Gamma}\ni\pi\mapsto A_{\pi}\in\mathcal{B}(\mathcal{H}_{\pi})$ such that
    \[\|A\|_{L^{\infty}(\widehat{\Gamma})}:=\sup_{\pi\in\widehat{\Gamma}}\|A_{\pi}\|_{\mathcal{B}(\mathcal{H}_{\pi})}<\infty,\]
    where the supremum is the essential supremum with respect to the Plancherel measure $\mu$. The space $L^{\infty}(\widehat{\Gamma})$ together the pointwise composition
    \[AB:\widehat{\Gamma}\ni\pi\mapsto A_{\pi}B_{\pi}\in\mathcal{B}(\mathcal{H}_{\pi}),\]
    and the norm $\|\cdot\|_{L^{\infty}(\widehat{\Gamma})}$ is a von Neumann algebra.
    \medskip

Let $\mathcal{B}_L(L^2(\Gamma))$ be the subspace of operators $T\in\mathcal{B}(L^2(\Gamma))$ which are left-invariant, i.e., such that $T\circ L_x=L_x\circ T$, for every $x\in\Gamma$, and where $L_x(f)(y):=f(x^{-1}y)$ is the left-translation operator.
\medskip

For each $\mu$-measurable field of uniformly bounded operators $A\in L^{\infty}(\widehat{\Gamma})$, define the operator $T_A\in\mathcal{B}_L(L^2(\Gamma))$ by
\[
T_A\phi:=\mathcal{F}^{-1}(A_{(\cdot)}\widehat{\phi}(\cdot)).
\]

It can be shown that the map $T:L^{\infty}(\widehat{\Gamma})\ni A\mapsto T_A\in \mathcal{B}_L(L^2(\Gamma))$ is a von Neumann algebra isomorphism and an isometry. Moreover, every operator $T\in\mathcal{B}_L(L^2(\Gamma)) $ arises in this way, namely, there exists a unique (up to zero measure sets with respect $\mu$) $\mu$-measurable field $A^{(T)}\in L^{\infty}(\widehat{\Gamma})$ such that for every $\phi\in L^2(\Gamma)$:
\begin{equation}\label{convdist}
\widehat{T\phi}(\pi)=A^{(T)}_{\pi}\widehat{\phi}(\pi)\qquad\mu\text{-almost everywhere}.
\end{equation}

Thus, we consider $\mathcal{K}(\Gamma)$ be the space of all the distributions $\kappa\in (C_c^{\infty})'(\Gamma)$ for which the convolution operator $C_c^{\infty}(\Gamma)\ni f\mapsto f*\kappa$ extends to a bounded operator $T_{\kappa}$ on $L^2(\Gamma)$. Thus, we define the Fourier transform of the distribution $\kappa$ by
\[
\widehat{\kappa}(\pi):=A^{(T_{\kappa})}_{\pi},\qquad\text{where }A^{(T_{\kappa})}\text{ is like in }\eqref{convdist}.
\]
Note that the convolution product between a function $f\in C_c^{\infty}(\Gamma)$ and a distribution $\kappa\in (C_c^{\infty}(\Gamma))'$ is given by
\[
f*\kappa(x):=\langle\kappa, L_{x^{-1}}f^*\rangle.
\]

In other words, $\widehat{\kappa}$ is defined in such a way that
\[
\mathcal{F}(T_{\kappa}f)(\pi)=\widehat{\kappa}(\pi)\widehat{f}(\pi),\qquad f\in L^2(\Gamma). 
\]

Over $\mathcal{K}(\Gamma)$ we can define an involution $\kappa^*(x):=\overline{\kappa(x^{-1})}$. For $\kappa_1,\kappa_2\in\mathcal{K}(\Gamma)$, consider the operators $T_{\kappa_1},T_{\kappa_2}\in B_L(L^2(\Gamma))$ as above and denote by $\kappa_1*\kappa_2$ its kernel. Also we can consider a norm in $\mathcal{K}(\Gamma)$ via the formula:
\[
\|\kappa\|_{\mathcal{K}(\Gamma)}:=\|f\mapsto f*\kappa\|_{\mathcal{B}(L^2)(\Gamma)}.
\]
Hence the space $\mathcal{K}(\Gamma)$ together with the involution, convolution product, and norm defined above is a von Neumann algebra isomorphic to $\mathcal{B}_L(L^2(\Gamma))$. In addition, the group Fourier transform $\mathcal{F}:\mathcal{K}(\Gamma)\to L^{\infty
}(\widehat{\Gamma})$ is an isometric isomorphism. 
\medskip

Denote by $\mathcal{M}(\Gamma)$ be the space of finite complex Borel measures in $\Gamma$. Then we have the following inclusions:
\[
L^1(\Gamma)\subset\mathcal{M}(\Gamma)\subset\mathcal{K}(\Gamma);
\]
this allows us to define the {\em group Fourier transform} of a finite complex Borel measure $\eta\in\mathcal{M}(\Gamma)$ by
    \begin{equation}\label{FTmeas}
        (\forall\pi\in\widehat{\Gamma}),\qquad\mathcal{F}(\eta)(\pi)=\widehat{\eta}(\pi):=\int_{\Gamma} \pi(x)d\eta(x),
        \end{equation}
    where the above integral is in the sense of Bochner.
%%%%%%%%%%%%%%%%%%%%%%%%%%%%%%%%%%%%%%%%%%%%%%%%%%%%%%%%%%%%%%%%%%%%%%%%%%%%%%%%%%%%%%%%%%%%%%%%%%%%%%%%%%%%%%%%%%%%%%%%%%%%%%%%%%%%%%%%%%%%%%%%
\section{Recovery of the discrete spectrum in nilpotent Lie groups}\label{ssec:recovnil}

Let $\Gamma$ be an $d$-dimensional lcsc nilpotent Lie group, $H\subset\Gamma$ be an {\em $r$-dimensional lattice}, $r\leq d$, which means that $H$ is an $r$-dimensional uniform subgroup of $\Gamma$, satisfying that $\mathsf{exp}^{-1}(H)$ is an additive subgroup of $\mathfrak{g}\cong\R^d$. Also, assume that $L\subset\Gamma$ is a lattice containing $H$. From \cite[Theorem 5.1.6]{corgreen}, there exists a Mal'cev basis $\{X_1,\ldots,X_d\}$ of $\mathfrak{g}$ with rational structure constants such that

\[
L=\mathsf{exp}(\Z X_1)\cdot\ldots\cdot\mathsf{exp}(\Z X_d)\subset\Gamma.
\]

For this basis, we have the natural identifications $\Gamma\cong\phi^{-1}(\Gamma)=\R^d$ and $L\cong\phi^{-1}(L)=\Z^d$ under the diffeomorphism $\phi^{-1}:\Gamma\to\R^d$ defined in \eqref{diffnil}. Thus, without loss of generality, we can assume that $\Gamma=(\R^d,\cdot)$ and $L=(\Z^d,\cdot)$, where the law group is given by the Baker-Campbell-Hausdorff formula. In addition, since $\mathsf{exp}^{-1}(H)$ is an additive subgroup of $\mathfrak{g}$, after a rearrangement of the basis elements if it is necessary, we have that there must exist $q_1,\ldots,q_r\in\Z\setminus\{0\}$ such that 
\[
\mathsf{exp}^{-1}(H)=q_1^{-1}\Z X_1+\ldots+q_r^{-1}\Z X_r\subset\mathfrak{g}.
\]

Then, after applying the polynomial form of $\phi:\R^d\to\Gamma$ given by \eqref{polcoord}, we have that there are positive integers $N_1,\ldots, N_r$ such that

\begin{equation}\label{lattpointnil}
     H=\bigoplus_{i=1}^r N_i\Z\subset\R^r,
\end{equation}

in which the identification is an group homomorphism when restricted to each direct summand.
\medskip

If $\mathsf{vol}_r$ denotes the Lebesgue measure in $\R^r$, and since $\mathsf{vol}( B_R^r)\sim R^r$ where $B_R^r:=[-R/2,R/2)^r$, we have that there exists a positive constant $\lambda_r=\lambda_r(H)$ such that
\begin{equation}\label{asympdenlat}
    \left||H\cap B_R^r|-\lambda_r\mathsf{vol}_r\left(B_R^r\right)\right|=o\left(\mathsf{vol}_r(B_R^r)\right)=o(R^r),
\end{equation}

Write $\delta_H:=\sum_{p\in H}\delta_p$ and keep the notations of the representation theory of lcsc nilpotent Lie groups from \Cref{ssec: Kirillov}. Following \eqref{FTmeas} and in view of \eqref{asympdenlat}, we define the {\em (averaged) Fourier transform of $H$} by

\begin{equation}\label{FTnil}
    \mathcal{F}(\delta_H)(\pi_l):=\lim_{R\to\infty}\frac{\mathcal{F}(\delta_{H\cap B_R^r})(\pi_l)}{R^r}=\lim_{R\to\infty}\frac{1}{R^r}\sum_{p\in H\cap B_R^r}\pi_l(p);
\end{equation}

more precisely, for $\varphi\in\mathcal{H}_l\cong L^2(\R^s)$ and $\psi\in C_c(\R^s)$, the following formula holds

\begin{equation}
    \langle\mathcal{F}(\delta_H)(\pi_l)\varphi,\psi\rangle_{L^2(\R^s)}=\lim_{R\to\infty}\frac{1}{R^r}\sum_{p\in H\cap B_R^r}\int_{\R^s}\varphi(px)\psi(x)dx.
\end{equation}
\medskip

Let $\xi_p, p\in H$ be i.i.d random vectors defined in a probability space $(\Omega,\mathcal{T},\mathbb{P})$ with $\xi_p\sim\xi$ and let $H_{\xi}$ be the realization set $\{p\cdot\xi_p:\ p\in H\}$. In this section, our goal is to obtain the recovery of the Fourier Transform $\widehat{\delta_H}$ from its random perturbations under suitable moment conditions over the perturbations $\xi_p$.

\begin{thm}\label{nilrecov}
Under the previous notations, assume that there is a positive number $\varepsilon$ such that $\mathbb{E}(|\mathsf{proj}_r(\xi)|^{r+\varepsilon})<\infty$, where $\mathsf{proj}_r:\R^d\to \R^r$ is the orthogonal projection onto the first $r$ coordinates. Then almost surely, for every $l\in\mathfrak{g}^*$ %for which $\mathsf{Stab}_l\neq\{\mathsf{id}\}$, there holds
\begin{equation}\label{recovnil}
   \mathcal{F}(\delta_{H_{\xi}})(\pi_l)=\mathcal{F}(\delta_H)(\pi_l)\mathbb{E}[\pi_{l}(\xi)],
\end{equation}

where $\mathcal{F}(\delta_{H_{\xi}})$ stands for the limit $\lim_{R\to\infty}\frac{1}{R^r}\mathcal{F}\left(\delta_{H_{\xi}\cap\left[-R/2,R/2\right)^d}\right)$. 
\end{thm}

In what follows, we are dedicated to the proof of \Cref{nilrecov}. Let $M, N$ be positive integers, $N_i|N$ for every $i=1,\ldots,r$, and consider the inclusions:
\begin{equation}\label{reslatt}
    N\mathbb{Z}\subset N_i\Z\subset\frac{1}{M}\Z,\qquad\text{for }1\leq i\leq r.
\end{equation}
By taking the quotient by $N\Z$, we find an embedding of the discrete group
\[
    H_N=\Gamma^r_N:=\bigoplus_{i=1}^r N_i\Z/N\Z\cong\bigoplus_{i=1}^r\Z/\frac{N}{N_i}\Z
\]
into
\[
    \Gamma_{M,N}^d:=\bigoplus_{i=1}^d\left(\frac{1}{M}\Z/N\Z\right),
\]
It should be noted that, by the rationality of $\R^d$ and $\R^r$, the rescalings $N\Z^d$ and $\frac{1}{M}\Z^d$ are still lattices in $\R^d$, and thus $H_N\subset\Gamma_{M,N}^d$ are indeed finite groups.
\medskip

To prove \Cref{nilrecov}, we will "discretize" the random perturbations $p\cdot\xi_p,p\in H$, restricted to $[-N/2,N/2)_{i=1}^d$ by approximating them by group elements taking values in a finer and finer subgrid 
\[
\frac{1}{M}\Z^d\cap\left[-\frac{N}{2},\frac{N}{2}\right).
\]
In the limit $M,N\to\infty$, after showing that the DFT converges (in a very precise sense to be stated below) to the usual Fourier transform, we will get the desired result as a consequence of an adequate quantitative recovery of the DFT (see \Cref{corobasic-nil}).

\subsubsection{Discretizing $\xi$}\label{ximnnil}
With the inclusions $N\Z\subset N_i\Z\subset\frac1M\Z\subset \mathfrak{g}$ in mind, for $i=1,\ldots,r$, we restrict and discretize the maps $\xi_{(\cdot)}:H\times\Omega\to\R^d$ defined immediately before \Cref{nilrecov}, to obtain $\xi^{(M,N)}_{(\cdot)}: H\times\Omega\to \Gamma^d_{M,N}$, as follows.
For $p\in H\cap\left[-\frac{N}{2},\frac{N}{2}\right)^d$ there exists a unique $p'\in H_N$ represented by $p$. 
Also, note that the inclusion $\frac1M\Z^d\subset\R^d$ passes to the quotient, and we get an inclusion 

\[\Gamma^d_{M,N}\subset \R^d/N\Z^d.\]

Therefore for $p'\in H_N$ we set $\xi_{p'}^{(M,N)}(\omega):=v\in\Gamma_{M,N}^d$ defined via the property that the $N\Z^d$-quotient of $\xi_p(\omega)$, denoted by $\xi_{p'}^{(N)}(\omega)$, falls into the cube 

\[v\cdot\left[-\frac{1}{2M},\frac{1}{2M}\right)^d.\] 

With these notations, let $(H_N)_{\xi^{(M,N)}}=\{p\cdot\xi_p^{(M,N)}:\ p\in H_N\}\subset\Gamma_{M,N}^d$.

\medskip

The following claims are direct to prove and are left to the reader:
\begin{claim}\label{disxi1}
    $\{\xi_p^{(M,N)}:\ p\in H_N\}$ are independent.
\end{claim}
\begin{claim}\label{limitxi}
    For each $p\in H_N$, $\xi_p^{(M,N)}(\omega)$ converges to $\xi_p(\omega)$ when $M,N\to\infty$.
\end{claim}
%%%%%%%%%%%%%%%%%%%%%%%%%%%%%%%%%%%%%%%%%%%%%%%%%%%%%%%%%%%%%%%%%%%%%%%%%%%%%%%%%%%%%%%%%%%%%%%%%%%%%%%%%%%%%%%%%%%%%%%%%%%%%%%%%%%%%%%%%%%%%%%%
\subsection{Discrete-to-continuum limits for irreducible representations in nilpotent Lie groups}\label{ssec: nildisc}

We start by considering a dual element $l\in\mathfrak{g}^*$, with rational coefficients, being the representative of an orbit $\mathcal{O}_l$. Without loss of generality, we write:
\begin{equation}\label{ratdual}
    l=\sum_{j=1}^d \frac{c_j}{m_j}X_j^*\qquad\text{with}\qquad c_j\in\Z\text{ and }\mathsf{gcd}(c_j,m_j)=1\text{ for every }j=1,\ldots,d,
\end{equation}

where some of the coefficients $c_j$ could be zero, and denote by $\mathfrak m=\mathfrak{m}_l$ its polarizing subalgebra; also write $K_l:=\mathsf{exp}(\mathfrak m)$.
\medskip

By \Cref{kirivmeth}, the map $\pi_l:\Gamma\to \mathcal{U}(\mathcal{H}_l)$ given by

\begin{equation}\label{nilrep}
    \pi_{l}:=\mathsf{ind}_{K_l}^{\Gamma}(\chi_{l}),\quad\text{where}\qquad\chi_{l}(\mathsf{exp}(Y))=e^{2\pi i l(Y)},
\end{equation}

is a unitary irreducible representation of $\Gamma$ and all the elements of $\widehat{\Gamma}$ arise in this way; also, we have that $\mathcal{H}_l$ is isomorphic to $L^2(\R^s)$, where $s=\mathsf{dim}(\Gamma/K_l)$. More precisely, the action of $\pi_l$ over $L^2(\R^s)$ is described as follows: since $K_l\subset \Gamma$ are Lie groups and $K_l$ is closed and normal in $\Gamma$, we can consider $\alpha:\Gamma/K_l\cong\R^s\to\Gamma$ be a continuous cross-section for the canonical continuous quotient map $i:\Gamma\to\Gamma/K_l$, like at the end of \Cref{ssec: indrep}. Let $k:\Gamma\to K_l\cong\R^{d-s}$ and $t:\Gamma\to\alpha (\Gamma/K_l)\cong\R^s$ be the continuous functions given by the factorization

\[
x=k(x)t(x).
\]
Then, for $x,y\in\Gamma=(\R^d,\cdot)$ and $f\in\mathcal{H}_l$ we have that

\begin{equation}\label{pilexp}
(\pi_l(x)f)(y)=f(xy)=f(k(xy)t(xy))=\chi_l(k)f(t)=e\left(\sum_{j=1}^{d-s}\frac{c_j k_j}{m_j}\right) f(t),
\end{equation}

with $e(x):=e^{2\pi ix}$,  $t=t(xy)\in\alpha(\Gamma/K_l)$ and where we identify $k=k(xy)\in K_l\subset\R^d$ with the vector $(k_1,\ldots,k_{d-s})$. Thus, after identifying $\mathcal{H}_l\cong L^2(\R^s)$ through the isometry $\mathcal{H}_l\ni f\mapsto f\circ\alpha\in L^2(\R^s)$, we have that $\pi_l$ is an irreducible unitary representation acting over $L^2(\R^s)$, and given by \eqref{pilexp}. 
\medskip

Since $K_l$ is rational, it is not difficult to see that the rescalings $\frac{1}{M}\Z^{d-s}$ and $N\Z^{d-s}$, for $M,N\in\N$, are still lattices in $K_l$. Moreover, since $K_l\cong\R^{d-s}$ is a normal subgroup of $\R^d$ and the law group given by the polynomial coordinates in \eqref{polcoord} is assumed to have integer coefficients, we obtain that $N\Z^{d-s}$ is normal in $\frac{1}{M}\Z^{d-s}$; therefore, it makes sense to consider the groups $\Gamma_{M,N}^d$, $\Gamma_{M,N}^{d-s}$ and $\Gamma_{M,N}^s$ as good approximants, in the Gromov-Hausdorff sense, for $(\R^d,\cdot)$, $K_l\cong\R^{d-s}$ and $\alpha(\R^d/K_l)\cong\R^s$, respectively, for $M,N\gg 1$.
\medskip 

Hence, in the process of approximating $\pi_l$ by a suitable sequence of irreducible unitary representations, we proceed as follows. Fix positive integers $m|M$ and $N\in\N$ and consider
\[c=(c_1,\ldots,c_{d-s})\in(\Z/mM\Z)^{d-s}\qquad\text{such that}\qquad \mathsf{gcd}(c_j,mM)=1\text{ for every }j=1,\ldots,d-s.
\]

Then we define the unitary representation $\pi_{M,N,m,c}$ of $\Gamma_{M,N}^d$ by:

\begin{equation}\label{approxrepnil}
    (\pi_{M,N,m,c}(p)f)(q):=e\left(\sum_{j=1}^{d-s}\frac{c_j k_j(pq)}{m}\right)f(t(pq)),\qquad \text{where }f\in\ell^2\left(\left(\frac{1}{M}\Z/m\Z\right)^s\right),
\end{equation}

where the vectors 
\[k(p)=(k_1(p),\ldots,k_{d-s}(p))\in\frac{1}{M}\Z^{d-s}\qquad \text{and}\qquad t(p)\in\frac{1}{M}\Z^s\] 

in the formula \eqref{approxrepnil}, with $p\in\Gamma_{M,N}^d$, are taken modulo $m\Z^{d-s}$ and $m\Z^s$, respectively.
\medskip

From \eqref{approxrepnil} we get that $\pi_{M,N,m,c}:\Gamma_{M,N}^d\to\mathcal{U}(\ell^2(\Gamma_{M,m}^s))$ is an irreducible unitary representation of $\Gamma_{M,N}^d$. Also observe that, although the map 
\[\Z^{d-s}\times\N\ni (c,m)\mapsto\pi_{M,N,m,c}\in\widehat{\Gamma}_{M,N}^d,\] 

where $m|M$ and the parameter $c$ in the index is taken modulo $mM\Z^{d-s}$, probably is not exhaustive and there are many other irreducible representations of $\Gamma_{M,N}^d$ which are not covered by the formula \eqref{approxrepnil}, the irreducible representations of the form $\pi_{M,N,m,c}$ are sufficient to "approximate" $\pi_l$; we develop this fact with more details in \Cref{nilDFT-FT} below.
\medskip

Following \Cref{ssec: DFT}, consider the Discrete Fourier Transform $\mathcal{F}_{M,N}:\ell^2(\Gamma_{M,N}^d)\to L^2(\widehat{\Gamma}_{M,N}^d)$. In particular, for $f\in\ell^2(\Gamma_{M,N}^d)$, the Fourier transform $\mathcal{F}_{M,N}(f)$ acts over an irreducible representation of the form $\pi_{M,N,m,c}$ like in \eqref{approxrepnil} as follows:

\[\mathcal{F}_{M,N}(f)(\pi_{M,N,m,c})=\frac{1}{M^dN^d}\sum_{p\in\Gamma_{M,N}^d}f(p)\pi_{M,N,m,c}(p):\ell^2\left(\Gamma_{M,m}^s\right)\to\ell^2\left(\Gamma_{M,m}^s\right).\]

To be able to pass information from $\mathcal{F}_{M,N}(\delta_{H_N})$ to $\mathcal{F}(\delta_{X})$ we need to use an appropriate discrete-to-continuum approximation, which we define next.

\begin{remark}[periodization convention]\label{rmk:periodization}
    In order to compare functions on a discrete group $ \Gamma_{M,N}^s$ with functions on $\R^s$, we use the following convention: in the following, given $M\in\mathbb{N}$ and a compactly supported function $\psi\in C_c(\R^s)$ with $\mathsf{supp}(\psi)\subset[-M/2,M/2)^s$, we still denote its $ M\Z^s$-periodization by $\psi$.
\end{remark}
 
\begin{defi}\label{weaknilconv}
    We say that a sequence of linear operators $(T_n)_{n\in\N}$, where $T_n\in\mathcal{B}\left(\ell^2(\Gamma_{M_n,m_n}^s)\right)$ for $n\geq 1$, with $M_n,m_n\to\infty$, converges to a linear operator $T\in\mathcal{B}(L^2(\R^s))$ if for every $\varphi,\psi\in C_c(\R^s)$ the following holds:

    \begin{equation}\label{weaknilconvlim}
        \lim_{n\to\infty}\langle T_n\varphi,\psi\rangle_{\ell^2\left(\Gamma_{M_n,m_n}^s\right)}=\langle T\varphi,\psi\rangle_{L^2(\R^s)},
    \end{equation}

    where for a sufficiently large $n\geq 1$, we identify $\varphi,\psi$ with its $m_n\Z^s$-periodizations like in \Cref{rmk:periodization}.
\end{defi}

Now, for $H_N:=\bigoplus_{i=1}^r N_i\Z/N\Z$, with $N_i|N$ for all $i=1,\ldots,r$, consider the realization set $(H_N)_{\xi^{(M,N)}}=H_{\xi^{(M,N)}}:=\{p\cdot\xi_p^{(M,N)}: p\in H_N\}$. The key claim to prove \Cref{nilrecov} is that the Fourier Transform of $X$ and $X_{\xi}$ are weak limits (in the sense of \Cref{weaknilconv}) of the Discrete Fourier Transforms $\mathcal{F}_{M,N}(\delta_{H})$ and $\mathcal{F}_{M,N}(\delta_{H_{\xi^{(M,N)}}})$, respectively. 

\begin{lemma}\label{nilDFT-FT}
 Let $l=\sum_{j=1}^d\lambda_j X_j^*\in\mathfrak{g}^*$, fix $k\in\N$, and consider the irreducible unitary representation $\pi_l$ acting over $\mathcal{H}_l\cong L^2(\R^s)$ given by the formula \eqref{pilexp}. For any pair of positive integers $m,N$, let
 
 \[\Gamma_{km,N}^d:=\left(\frac{1}{km}\Z/N\Z\right)^d, \qquad H:=\bigoplus_{i=1}^r N_i\Z\subset\R^d,\qquad H_N:=\bigoplus_{i=1}^r N_i\Z/N\Z\subset\Gamma_{km,N}^d.\]
 
 For each $m\geq 1$ let $c_m:=(c_m^{(1)},\ldots,c_m^{(d-s)})\in (\Z/km^2\Z)^{d-s}$ such that $c_m^{(j)}/m\to\lambda_j$ when $m\to\infty$, for every $1\leq j\leq d-s$. Also consider the irreducible representations $\pi_{km,N,m,c_m}$ of $\Gamma_{km,N}^d$ like in \eqref{approxrepnil}. Then the following limit holds:
    \begin{equation}
         \lim_{N\to\infty}\lim_{m\to\infty}(km)^{d-s}N^{d-r}\mathcal{F}_{km,N}(\delta_{H_N})(\pi_{km,N,m,c_m})=\mathcal{F}(\delta_H)(\pi_l)
    \end{equation}
        
where the limit for $m\to\infty$ is taken in the sense of \Cref{weaknilconv}, and the weak limit for $N\to\infty$ is in the classical sense.  
\end{lemma}

\begin{proof}
     First we will show, in the sense of \Cref{weaknilconv}, that 

    \begin{equation}
     \begin{split}
         \lim_{m\to\infty}(km)^{d-s}N^{d-r}\mathcal{F}_{km,N}(\delta_H)(\pi_{km,N,m,c_m})=\frac{1}{N^r}\mathcal{F}\left(\delta_{H\cap \left[-\frac{N}{2},\frac{N}{2}\right)^r}\right)(\pi_l);
     \end{split}    
    \end{equation} 
     
     then, after taking $N\to\infty$ we will get the conclusion of \Cref{nilDFT-FT}. For this, fix $N\in\N$ and consider two non-zero functions $\varphi,\psi\in\mathcal{H}_l\cong L^2(\R^s)$ supported on $[-N/2,N/2)^s$; recall that by \Cref{rmk:periodization}, its $N\Z^s$-periodizations are still denoted by $\varphi,\psi$. Then by definition of the inner product $\langle\cdot,\cdot\rangle_{\ell^2(\Gamma_{km,m}^s)}$ and after the identifications 
     
     \[H_N\cong \bigoplus_{i=1}^r N_i\Z\cap\left[-\frac{N}{2},\frac{N}{2}\right)^r,\qquad\Gamma_{km,m}^s \cong\frac{1}{km}\Z^s\cap\left[-\frac{m}{2},\frac{m}{2}\right)^s,\] 
     we get that   

    \begin{equation*}
    \begin{split}
        & (km)^{d-s}N^{d-r}\langle\mathcal{F}_{km,N}(\delta_{H_N})(\pi_{km,N,m,c_m})\varphi,\psi\rangle_{\ell^2(\Gamma_{km,m}^s)}\\
        & =\frac{1}{(km)^sN^r}\sum_{p\in\Gamma_{km,m}^s}(\mathcal{F}_{km,N}(\delta_{H})(\pi_{km,m,c_m})\varphi)(p)\overline{\psi(p)}\\
        & =\frac{1}{(km)^sN^r}\sum_{p\in\frac{1}{km}\Z^s\cap\left[-\frac{m}{2},\frac{m}{2}\right)}\left\{\sum_{\substack{q\in\bigoplus_{i=1}^rN_i\Z\\ q\in \left[-\frac{N}{2},\frac{N}{2}\right)^r}}e\left(\sum_{j=1}^{d-s}\frac{c_m^{(j)}k_j(pq)}{m}\right)\varphi(t(pq))\overline{\psi(p)}\right\}\\
        & =\frac{1}{N^r}\sum_{\substack{q\in\bigoplus_{i=1}^r N_i\Z\\ q\in \left[-\frac{N}{2},\frac{N}{2}\right)^r}}\frac{1}{(km)^s}\left\{\sum_{p\in\frac{1}{km}\Z^s\cap\left[-\frac{m}{2},\frac{m}{2}\right)}e\left(\sum_{j=1}^{d-s}\frac{c_m^{(j)}k_j(pq)}{m}\right)\varphi(t(pq))\overline{\psi(p)}\right\}\\
        & \longrightarrow\frac{1}{N^r}\sum_{\substack{q\in\bigoplus_{i=1}^r N_i\Z\\ q\in \left[-\frac{N}{2},\frac{N}{2}\right)^r}}\int_{\left[-\frac{N}{2},\frac{N}{2}\right)^s}e\left(\sum_{j=1}^{d-s}\lambda_j k_j(xq)\right)\varphi(t(xq))\overline{\psi(x)}dx\\
        & =\frac{1}{N^r}\left\langle\mathcal{F}\left(\delta_{H\cap\left[-\frac{N}{2},\frac{N}{2}\right)^r}\right)(\pi_l)\varphi,\psi\right\rangle_{L^2(\R^s)},
    \end{split}
    \end{equation*}

    where the limit taken for $m\to\infty$ in the sense of \Cref{weaknilconv}. Hence the the conclusion of \Cref{nilDFT-FT} follows by the definition of $\pi_l$ and $\mathcal{F}(\delta_{H})$ after considering $\varphi,\psi\in L^2(\R^s)$, testing against $\varphi_N,\psi_N\in C_c(\R^s)$ being the $N\Z^s$-periodizations of $\varphi,\psi$ restricted to $\left[-\frac{N}{2},\frac{N}{2}\right)^s$, respectively, and taking the limit when $N$ goes to $\infty$. 
\end{proof}

%Observe that in the previous Lemma, the dual element $l_M\in\mathfrak{g}^*$ with rational coefficients plays the same role in $\pi_{M,N,l_M}$ as the coefficient $c_m$ in the representation $\pi_{0,0,c_m,m}$ in \Cref{DFT-FT-Heis} and \Cref{lemmalimxiDFT-Heis}.

\begin{lemma}\label{lemmalimxiDFT-nil}
    Under the assumptions of \Cref{nilDFT-FT}, suppose there exists a positive number $\varepsilon$ such that $\mathbb{E}[|\mathsf{proj}_r(\xi)|^{r+\varepsilon}]<\infty$.  Then the following limit holds almost surely :
    \begin{equation}
    \begin{split}
        \lim_{N\to\infty}\lim_{m\to\infty}(km)^{d-s}N^{d-r}\mathcal F_{km,N}\left(\delta_{H_{\xi^{(km,N)}}}\right)(\pi_{km,N,m,c_m})&=\mathcal{F}(\delta_{X_{\xi}})(\pi_l).
    \end{split}
    \end{equation}
\end{lemma}

\begin{proof}
Firstly consider the sets $A_N,A_N'\subset X$ given by 
    \begin{equation*}
    \begin{split}
        A_{N}&=\left\{p\in X\cap [-N/2,N/2)_{i=1}^r:\ p\cdot\xi_{p}\not\in [-N/2,N/2)_{i=1}^d \right\}\\
        A_{N}'&=\left\{p\in X\setminus[-N/2,N/2)_{i=1}^r:\ p\cdot\xi_{p}\in [-N/2,N/2)_{i=1}^d \right\}.
    \end{split}
    \end{equation*}
    Since $\mathbb E[|\mathsf{proj}_r(\xi)|^{r+\varepsilon}]<\infty$, by the same reasoning as Lemma 3.1 in \cite{Yakir} the following limits hold almost surely:
    \begin{equation}\label{escapemass}
        \lim_{N\to\infty}\frac{|A_N|}{N^r}=\lim_{N\to\infty}\frac{| A_N'|}{N^r}=0.
    \end{equation}
    
 Like in the limit in \Cref{nilDFT-FT}, given a positive integer $N$ and $\varphi,\psi\in C_c(\mathbb{R}^s)$ we need to estimate the value of: 
    \begin{equation}\label{2ndlimitDFT}
    \begin{split}
        &\frac{1}{N^r}\left|\left\langle\mathcal{F}\left(\delta_{H_{\xi}\cap\left[-\frac{N}{2},\frac{N}{2}\right)^d}\right)(\pi_l)\varphi,\psi\right\rangle_{L^2(\mathbb{R}^s)}-\frac{1}{(km)^s}\langle\mathcal{F}_{km,N}(\delta_{H_{\xi^{(km,N)}}})\varphi,\psi\rangle_{\ell^2\left(\Gamma_{km,m}^s\right)}\right|\\
        &=\frac{1}{N^r}\left|\int_{\R^s}\sum_{p\in H_{\xi}\cap\left[-\frac{N}{2},\frac{N}{2}\right)^d}e\left(\sum_{j=1}^{d-s}\lambda_j k_j\right)\varphi(t)\overline{\psi(x)}dx-\sum_{q\in \Gamma_{km,m}^s}\sum_{p\in H_{\xi^{(km,N)}}}e\left(\sum_{j=1}^{d-s}\frac{c_m^{(j)} k_j}{m}\right)\varphi(t)\frac{\overline{\psi(q)}}{(km)^s}\right|,
    \end{split}    
    \end{equation}
    where in the second line of \eqref{2ndlimitDFT}, for the first term we have that $k=k(px)=(k_1,\ldots,k_{d-s}), t=t(px), p\in H_{\xi}\cap\left[-\frac{N}{2},\frac{N}{2}\right)^d$, and for the second term $k=k(pq), t=t(pq), p\in H_{\xi^{(km,N)}}$. 
    \medskip
    
    Now observe that 
    
    \begin{equation}\label{FZxidecomp}
    \begin{split}
        &\int_{\R^s}\sum_{p\in H_{\xi}\cap\left[-\frac{N}{2},\frac{N}{2}\right)^d}e\left(\sum_{j=1}^{d-s}\lambda_j k_j(px)\right)\varphi(t(px))\overline{\psi(x)}dx\\
        &=\int_{\R^s}\sum_{p\in H\cap\left[-\frac{N}{2},\frac{N}{2}\right)^d}e\left(\sum_{j=1}^{d-s}\lambda_j k_j(p\cdot\xi_p\cdot x)\right)\varphi(t(p\cdot\xi_p \cdot x))\overline{\psi(x)}dx+S_N-S_N',
    \end{split}
    \end{equation}
    where $S_N$ and $S_N'$ are defined as the integral in the right-hand side in \eqref{FZxidecomp}, but with $p$ running over $A_N$ and $A_N'$, respectively.
    \medskip
    
    Since the function $e(\cdot)=e^{-2\pi i(\cdot)}$ is bounded, and $\varphi,\psi$ are compactly supported, then from \eqref{escapemass} we have that almost surely $\frac{S_N}{N^r},\frac{S_N'}{N^r}\to 0$ when $N$ goes to $\infty$.
    Then from \Cref{limitxi}, \eqref{FZxidecomp}, \eqref{escapemass}, and by the triangle inequality, given $\varepsilon>0$ there are $M_0,N_0\in\mathbb{N}$ such that for every $M\geq M_0, N\geq N_0$ we get that
    \begin{equation*}\label{DFT-FT-xi}
    \begin{split}
        %&\frac{1}{N^r}\left|\int_{\R^s}\sum_{p\in\mathbb{Z}^d\cap\left[-\frac{N}{2},\frac{N}{2}\right)^d}\varphi(p\cdot\xi_p\cdot x)\overline{\psi(x)}dx-\sum_{k\in\Gamma_{M,N}^s}\sum_{p\in H} \varphi(p\cdot\xi_p^{(M,N)}\cdot k)\frac{\overline{\psi(k)}}{M^s}\right|+\frac{|S_N|}{N^d}+\frac{|S_N'|}{N^d}\\
        \eqref{2ndlimitDFT}
        &\leq\frac{1}{N^r}\sum_{p\in H\cap\left[-\frac{N}{2},\frac{N}{2}\right)^r}\left|\int_{\left[-\frac{N}{2},\frac{N}{2}\right)^s}e\left(\sum_{j=1}^{d-s}\lambda_j k(p\cdot\xi_p x)\right)\varphi(p\cdot\xi_p\cdot x)\overline{\psi(x)}dx\right.
        \\
        &-\left.\sum_{\substack{q\in\frac{1}{km}\Z^s\\ q\in\left[-\frac{m}{2},\frac{m}{2}\right)^s}}e\left(\sum_{j=1}^{d-s}\frac{c_m^{(j)} k_j(p\cdot\xi_p q)}{m}\right)\varphi( p\cdot \xi_p^{(km,N)}\cdot q)\frac{\overline{\psi(q)}}{(km)^s}\right|+\frac{|S_N|}{N^r}+\frac{|S_N'|}{N^r}\\
        &\leq\frac{1}{N^r}\frac{\varepsilon N^r}{3}+\frac{\varepsilon}{3}+\frac{\varepsilon}{3}=\varepsilon.
    \end{split}
    \end{equation*}
    This concludes the proof of \Cref{lemmalimxiDFT-nil}.
\end{proof}

The final step to prove \Cref{recovnil} is given by the next quantitative recovery of the discrete Fourier spectrum, which is a direct consequence of \Cref{basicthmH}.

\begin{prop}\label{corobasic-nil}
    Let $k\in N$ be fixed, $N,m\in\mathbb{N}$ and consider the finite groups $\Gamma_{km,N}^d$ and $H_{N}$ as before. Let $\xi_p^{(km,N)},p\in H_N$ and $\varepsilon>0$ satisfying the hypotheses of \Cref{basicthmH}. Then with probability at least $1-\frac{C}{\varepsilon^2|H_N|}$, for some suitable positive constant $C=C(H)$ there holds:
    \begin{equation}\label{asbasic-nil}
        \left\|\mathcal{F}_{km,N}(\delta_{H_{\xi^{(km,N)}}})-\mathcal{F}_{km,N}(\delta_{H})\mathbb{E}(\xi^{(km,N)}(\cdot))\right\|^2_{L^2(\widehat{\Gamma}_{km,N}^d)}\leq\frac{1}{(km)^{2d}N^{2d-4r}}\left(\varepsilon+\frac{4C^2}{N^r}\right)\max_{\pi\in\widehat{\Gamma}_{km,N}^d}d_{\pi}.  
    \end{equation}
\end{prop}
\begin{proof}
    This follows directly from \Cref{basicthmH} and from the fact that $|H_N|=\frac{N^r}{N_1\ldots N_r}$ and $|\Gamma_{km,N}^d|=(km)^dN^d$.
\end{proof}

To prove \Cref{nilrecov}, we need a version of \eqref{asbasic-nil} in the sense of \Cref{weaknilconv}, namely
\begin{equation}\label{FTasconv}
    M^{d-s}N^{d-r}\left(\mathcal{F}_{M,N}(\delta_{H_{\xi^{(M,N)}}})(\pi_{M,N,m,c})-\mathcal{F}_{M,N}(\delta_{H})\mathbb{E}\left(\pi_{M,N,m,c}\left(\xi^{(M,N)}\right)\right)\right)\longrightarrow 0,
\end{equation}
where $\pi_{M,N,m,c}$ is the finite dimensional representation having the form of \eqref{approxrepnil}. This follows from \eqref{basicthmHpw}, since for $\varepsilon>\left(\frac{48 C^2}{N^r}\right)^{1/2}$, with probability at least $1-\frac{48 C^2}{\varepsilon^2N^r}$ the following error bound holds for every $\pi\in\widehat{\Gamma}_{km,N}^d$: 

\begin{equation}\label{dualnorm-l2}
    \left\|\left(\mathcal{F}_{km,N}(\delta_{H_{\xi^{(km,N)}}})(\pi)-\mathcal{F}_{km,N}(\delta_H)(\pi)\mathbb{E}(\pi(\xi^{(km,N)}))\right)e_i^{(\pi)}\right\|_{\Gamma_{km,N}^s}^2\leq \frac{N^{4r}}{(km)^{2d}N^{2d}}\left(\varepsilon+\frac{4C^2}{N^r}\right)d_{\pi}^2,
\end{equation}

where $(e_i^{(\pi)})_{i=1}^{d_{\pi}}$ is an orthonormal basis for $V_{\pi}$. Then after multiplying \eqref{dualnorm-l2} by $M_m^{(d-s)/2}N^{(d-r)/2}$ with $M_m=km,\ \pi=\pi_{km,N,n,c}$ and $d_\pi=(km)^sm^s$ as before it follows that:

\begin{equation}\label{asest2}
    \begin{split}
        &\left\|\left((km)^{d-s}N^{d-r}\mathcal{F}_{km,N}(\delta_{H_{\xi^{(km,N)}}})(\pi)-(km)^{d-s}N^{d-r}\mathcal{F}_{km,N}(\delta_H)(\pi)\mathbb{E}(\pi(\xi^{(km,N)}))\right)e_i^{(\pi)}\right\|_{\Gamma_{km,N}^s}^2\\
        &\leq\frac{1}{k^{\frac{3d}{2}-s}N^{\frac{3d-9r}{2}}}\cdot\frac{1}{m^{\frac{3d}{2}-3s}}\left(\varepsilon+\frac{4C^2}{N^r}\right).%\longrightarrow 0\qquad\text{when } m\to\infty,
    \end{split}
\end{equation}

\begin{proof}[Proof of \Cref{nilrecov}]
    On the one hand, from \Cref{nilDFT-FT} and \Cref{lemmalimxiDFT-nil}, we have the following limit holds in the sense of \Cref{weaknilconv}:
    \begin{equation}
        \begin{split}\label{limitproof1}
            &\lim_{m\to\infty}(km)^{d-s}N^{d-r}(\mathcal{F}_{km,N}(\delta_{H_{\xi^{(km,N)}}})(\pi_{km,N,m,c_m})-\mathcal{F}_{km,N}(\delta_H)(\pi_{km,N,m,c_m})\mathbb{E}(\pi_{km,N,m,c}(\xi^{(km,N)})))\\
            &=\frac{1}{N^r}\mathcal{F}\left(\delta_{H_{\xi}\cap\left[-\frac{N}{2},\frac{N}{2}\right)^d}\right)(\pi_l)-\frac{1}{N^r}\mathcal{F}\left(\delta_{H\cap\left[-\frac{N}{2},\frac{N}{2}\right)^d}\right)(\pi_l)\mathbb{E}(\pi_l(\xi)).
        \end{split}
    \end{equation}

On the other hand, from \eqref{asest2} and \Cref{disxi1}, for $\varepsilon>\left(\frac{48C^2}{N^r}\right)^{1/2}$, with probability at least $1-\frac{48C^2}{\varepsilon^2 N^r}$ we have that 

\begin{equation}\label{limitproof2}
\begin{split}
    (km)^{d-s}N^{d-r}(\mathcal{F}_{km,N}(\delta_{H_{\xi^{(km,N)}}})(\pi_{km,N,m,c})-\mathcal{F}_{km,N}(\delta_H)(\pi_{km,N,m,c})\mathbb{E}(\pi_{km,N,m,c}(\xi^{(km,N)})))\to 0,
\end{split} 
\end{equation}

where again the limit for $m\to\infty$ is taken in the sense of \Cref{weaknilconv}, and we use \Cref{lemma:dimirrednil} in \eqref{asest2} to ensure that $3d/2-3s>0$. Therefore, by taking $\varepsilon=\varepsilon_N\to 0$ such that $1/|H_N|\ll\varepsilon_N$ and putting together \eqref{limitproof1},\eqref{limitproof2}, there holds that almost surely

\[
\lim_{N\to\infty}\frac{1}{N^r}\mathcal{F}\left(\delta_{H_{\xi}\cap\left[-\frac{N}{2},\frac{N}{2}\right)^d}\right)(\pi_l)=\left(\lim_{N\to\infty}\frac{1}{N^r}\mathcal{F}\left(\delta_{H\cap\left[-\frac{N}{2},\frac{N}{2}\right)^d}\right)(\pi_l)\right)\mathbb{E}(\pi_l(\xi)),
\]

where the limits are taken like in \eqref{FTnil}, and thus the conclusion of \Cref{nilrecov} follows.
\end{proof}
%%%%%%%%%%%%%%%%%%%%%%%%%%%%%%%%%%%%%%%%%%%%%%%%%%%%%%%%%%%%%%%%%%%%%%%%%%%%%%%%%%%%%%%%%%%%%%%%%%%%%%%%%%%%%%%%%%%%%%%%%%%%%%%%%%%%%%%%%%%%
\section{Examples}\label{sec: examples}

In this section, we illustrate with some examples the quantitative and almost sure recovery given by \Cref{basicthmH} and \Cref{nilrecov}. We mainly focus on the two canonical instances of abelian and non-abelian groups, namely, the Euclidean space and the Heisenberg group, where some explicit error bounds are provided. In addition, in order to clarify the computations in the setting of general finite groups, we decide to incorporate the corresponding ones for the Euclidean case.

\subsection{Finite abelian groups}\label{ssection: setting} 

We consider a discrete torus $\mathbb T_N^d:=(\mathbb Z/N\mathbb Z)^d$ where $N$ is not a prime number, and its subset $H=(M_1 \mathbb Z/ N\mathbb Z)\oplus\ldots\oplus(M_r\mathbb{Z}/N\mathbb{Z})$ for $r\leq d$ and for some divisor $M_i|N$ with $1\leq i\leq r$. The general case for the study in an ambient group $(\mathbb Z/ N_1\mathbb Z)\oplus\ldots\oplus(\mathbb{Z}/N_d\mathbb{Z})$ instead $\mathbb{T}_N^d$ follows analogously; we restrict ourself to $\mathbb{T}_N^d$ only to simplify notations and computations. We denote by $\mathcal F:\ell^2(\mathbb T_N^d)\to \ell^2(\mathbb T_N^d)$ the Discrete Fourier Transform (see \cite[Ch. 2]{terras} for a treatment of the $d=1$ case, and the case of general $d$ is an easy extension), given by
\begin{equation}\label{dftdef}
 \mathcal{F}(f)(k)=\widehat f(k):=\frac{1}{N^d}\sum_{p\in \mathbb T_N^d} e(p,k) f(p),\quad\text{where}\quad e(p,k)=e_N(p,k):=e^{\frac{-2\pi}Ni \langle p, k\rangle}.
\end{equation}
%
%The scaling factor $1/(M_1\ldots M_r)= |X|/ |\mathbb T_N^r|$ is justified because we think of $X$ to have unit density within the subspace $\mathbb T_N^r\subseteq\mathbb T_N^d$ that it generates.

\medskip

We apply \eqref{dftdef} to the characteristic function
\[
    \delta_H(p):=\left\{
        \begin{array}{ll}
            1&\text{ if }p\in H,\\[2mm]
            0&\text{ else,}
        \end{array}\right.
\]
and then we see the characteristic function of the dual set $H^*$ of $H$ appears. Precisely, we have
\begin{equation*}
\begin{split}
    \widehat\delta_H=\frac{1}{N^{d-r}(M_1\ldots M_r)}\delta_{H^*},\quad\text{with }\quad H^*
    &:=\{k\in \mathbb T_N^d:\ \forall p\in H,\ k\cdot p\equiv 0\ (\text{mod }N)\}\\
    &=\bigoplus_{i=1}^r\left(\frac{N}{M_i}\mathbb Z/ N\mathbb Z\right)\oplus(\mathbb{Z}/N\mathbb{Z})^{d-r}.
 \end{split}
\end{equation*}
Note that $|H|=\frac{N^r}{(M_1\ldots M_r)}$ and $|H^*|=(M_1\ldots M_r)N^{d-r}$.

\medskip

We then consider a random perturbation field of a configuration $H$ as above, namely we fix a probability space $(\Omega, \mathcal T, \mathbb P)$ and define a collection of $\mathbb T_N^d$-valued random variables $\xi=(\xi_p)_{p\in H}$, and the $\xi$-perturbed version of $H$ will be denoted by $H_\xi=H_\xi(\omega):=\{p+\xi_p(\omega):\ p\in H\}$, which is a random subset of $\mathbb T_N^d$, with $\omega\in\Omega$. We here consider that $\xi_p, p\in H$ are i.i.d. random variables. 
\medskip

The main quantitative recovery result in the framework of finite groups given by \Cref{basicthmH} reads as follows.

\begin{thm}\label{basicthm}
Let $\xi_p, p\in H$ be i.i.d. random variables with $\xi_p\sim\xi$ and $C:=\sup_{\lambda\in \mathbb T_N^d}\mathsf{Var}[e(\xi,\lambda)]<\infty$, and assume that $\varepsilon>\sqrt{\frac{72(M_1\ldots M_r)}{N^r}}$ and $\frac{N^r}{M_1\ldots M_r}>12$. Then with probability at least $1-\frac{72(M_1\ldots M_r)}{\varepsilon^2 N^r}$ we have
\begin{equation}\label{mainerror}
    \max_{\lambda\in \mathbb T_N^d}\left|\widehat{\delta_{H_\xi}}(\lambda) - \mathbb E\left[e(\xi,\lambda)\right]\ \widehat{\delta_H}(\lambda)\right|\le \frac{ 1}{|H^*|}\left(\varepsilon+\frac{4C^2}{|H|}\right)^{\frac{1}{2}}.%=\frac{1}{N^{d-r}(M_1\ldots M_r)}\left(\varepsilon+\frac{4C^2 (M_1\ldots M_r)}{N^r}\right)^{\frac{1}{2}}.
\end{equation}
\end{thm}
Observe that $C\le 4$ due to the boundedness of $e(\cdot,\cdot)$. The regime $(M_1\ldots M_r)^{1/r}\ll N$ and the choice $\left(\frac{M_1\ldots M_r}{N^r}\right)^{1/2}\ll \varepsilon\lesssim 1$ allow to obtain "good recovery with small error", and is thus the most interesting for recovery problems, in the following sense:
\begin{corollary}\label{corobasic}
Consider a sequence of $N^{(n)},M^{(n)}_1,\ldots M_r^{(n)}$ with corresponding torii $\mathbb T_N^d=\mathbb T_{N^{(n)}}^d$ and subsets $H=H^{(n)}$, and $\xi_p^{(n)},p\in H^{(n)}$ satisfying the hypotheses of \Cref{basicthm}, with $\xi_p^{(n)}$ independent from each other for different values of $n$.

\medskip

If we have

\[
    (M_1^{(n)}\ldots M_1^{(n)})^{1/r}\ll N^{(n)},
\]
then almost surely
\begin{equation}\label{asbasic}
    \left\|\widehat{\delta_{H_{\xi^{(n)}}}}- \mathbb E[e(\xi^{(n)},\cdot)]\widehat{\delta_{H^{(n)}}}\right\|_{\ell_\infty(\mathbb T_N^d)}\to 0.
\end{equation}
\end{corollary}
Given definition \eqref{dftdef}, it is natural in the pursuit of a bound \eqref{mainerror} to consider the functions $f, g:H\times\mathbb T_N^d\to\mathbb C$ defined by as
\begin{equation}\label{defab}
    f(p,\lambda):=e(p,\lambda),\quad g(p,\lambda):=\overline{e(\xi_p, \lambda) - \mathbb E[e(\xi_p,\lambda)]]}%=\overline{e(\xi_p, \lambda) - \varphi_{\xi_p}(\lambda)}.
\end{equation}
Note that $g$ is a random function, as it depends on the events $\omega$ through the $\xi_p$'s, whereas $f$ is deterministic. With definition \eqref{defab} we restate the desired bound \eqref{mainerror} as
\begin{equation}\label{mainerror2}
    \mathbb P\left\{\frac{\left\|\langle f_{\lambda}, g_{\lambda}\rangle_{\ell^2(H)}\right\|_{\ell^\infty(\mathbb T_N^d)}}{N^d}> 
    \frac{ 1}{|H^*|}\left(\varepsilon+\frac{4C^2}{|H|}\right)^{1/2}\right\}\le \frac{72 (M_1\ldots M_r)}{\varepsilon^2N^r},
\end{equation}
where we also use the notation $f_\lambda:=f(\cdot,\lambda), g_\lambda:=g(\cdot,\lambda)$. We note that $\langle f_\lambda,g_\lambda\rangle_{\ell^2(H)}=f_\lambda*g_\lambda(0)$, in which we use the convolution of functions $\phi,\psi$ over $\mathbb T_N^d$, defined by 
\[
    \phi*\psi(x):=\sum_{y\in\mathbb T_N^d}\phi(y)\psi(x-y).
\]
The function $f_\lambda*g_\lambda$ is supported on $H+H=H$ (the sum taking place in the group $\mathbb T_N^d=(\mathbb Z/N\mathbb Z)^d$) and therefore
\begin{equation}\label{fgphi}
    \langle f_\lambda, g_\lambda\rangle_{\ell^2(H)}=\sum_{y\in H}f_\lambda*g_\lambda (y) \delta_0(y),
\end{equation}
where $\delta_0:=\delta_H$ for the special choice $H=\{0\}$.

\medskip 

We extend the functions $f_\lambda,g_\lambda$ to $F_\lambda,G_\lambda:\mathbb T_N^d\times\mathbb T_N^d\to \mathbb R$ by zero outside $H$. Then we have
\[
    \langle f_\lambda, g_\lambda\rangle_{\ell^2(H)}=\langle F_\lambda, G_\lambda\rangle_{\ell^2(\mathbb T_N^d)},\quad F_\lambda*G_\lambda(y)=\delta_H(y)f_\lambda*g_\lambda(y).
\]
With our normalizations, we check that 
\begin{eqnarray}\label{DFTconvtorus}
    \widehat{F*G}(k)=\frac{1}{N^d}\sum_{q\in \mathbb T_N^d}(F*G)(q)e(q,k)&=&\frac1{N^d}\sum_{p,q\in \mathbb T_N^d}F(p)G(q-p)e(p,k)e(q-p,k)\nonumber\\
    &=& N^d\widehat F (k)\widehat G(k),
\end{eqnarray}
and that
\begin{eqnarray}\label{parseval}
    \langle F, G\rangle_{\ell^2(\mathbb T^d_N)}&=&\sum_{p\in\mathbb T_N^d}F(p)\overline{G(p)}\nonumber\\
    &=&\frac{1}{N^d}\sum_{\substack{p\in\mathbb T_N^d\\ k\in\mathbb T_N^d}}F(p)e(p,k)\overline{G(p)e(p,k)}=\frac{1}{N^d}\sum_{\substack{p,p'\in\mathbb T_N^d\\ k\in\mathbb T_N^d}}F(p)e(p,k)\overline{G(p')e(p',k)}\nonumber\\
    &=& N^d\sum_{k\in\mathbb T_N^d}\widehat F(k)\overline{\widehat G(k)}=N^d\langle \widehat F, \widehat G\rangle_{\ell^2(\mathbb T^d_N)}.
\end{eqnarray}
Combining \eqref{fgphi}, \eqref{DFTconvtorus} and \eqref{parseval}, we get (observing that $\widehat\delta_0\equiv N^{-d}$ by our definition)
\begin{equation*}
    \langle f_\lambda, g_\lambda\rangle_{\ell^2HX)}=\langle F_{\lambda}*G_{\lambda},\delta_0\rangle_{\ell^2(\mathbb T_N^d)}=N^d\sum_{x\in\mathbb T_N^d}\widehat{F_\lambda}(x)\overline{\widehat{G_\lambda}(x)}.
\end{equation*}
Note that, for $x\in \mathbb T_N^d$:
\begin{equation}\label{suppFlamb}
    \widehat{F_\lambda}(x)=\frac{1}{N^d}\sum_{p\in H}e(p,\lambda)e(p,x)=\widehat{\delta_H}(\lambda+x)=\frac{1}{N^{d-r}(M_1\ldots M_r)}\delta_{H^*-\lambda}(x).
\end{equation}
Therefore we get
\begin{equation}\label{hatfhatgerror}
    \langle f_\lambda, g_\lambda\rangle_{\ell^2(H)}=\frac{N^r}{M_1\ldots M_r}\sum_{y\in H^*-\lambda}\overline{\widehat G_\lambda(y)}.
\end{equation}
We note that the $r=d$ case of \eqref{hatfhatgerror} simplifies continuum-case results such as \cite[Theorem A.4]{pv}, or of \cite[Appendix]{Yakir}.

\medskip
We now study the terms $\widehat{G_\lambda}(y)$ from \eqref{hatfhatgerror}. Denoting $\widetilde{\varphi}(x):=\overline{\varphi(-x)}$, we rewrite the autocorrelation $\gamma_\lambda$ of the random functions $g(\lambda,q)=g_\lambda(q)$ as follows:
\begin{equation}\label{autocorrel}
    \gamma_\lambda(q):=\frac{1}{|H|}G_\lambda *\widetilde{G_\lambda}(q)=\frac{M_1\ldots M_r}{N^r}G_\lambda *\widetilde{G_\lambda}(q)=\left\{\begin{array}{ll}\frac{M_1\ldots M_r}{N^r}\sum_{p\in H}|g_\lambda(p)|^2&\text{ if }q=0,\\[3mm]
    \frac{M_1\ldots M_r}{N^r}\sum_{p\in H}g_\lambda(p)\overline{g_\lambda(p+q)}&\text{ else.}\end{array}\right.
\end{equation} 
From \eqref{DFTconvtorus} we also obtain, for $x\in\mathbb T_N^d$,
\begin{equation}\label{eqgamma}
    |\widehat{G_\lambda}|^2(x)=\frac1{N^d}\widehat{G_\lambda * \widetilde{G_\lambda}}(x)=\frac{1}{N^{d-r}(M_1\ldots M_r)}\widehat{\gamma_\lambda}(x).
\end{equation}

\begin{prop}\label{variance} 
    Let $\varepsilon>\left(\frac{72 M_1\ldots M_r}{N^r}\right)^{1/2}$
    and assume that $\frac{N^r}{M_1\ldots M_r}>12$. 
    Under our hypotheses on $\xi$, with probability at least $1-\frac{72}{\varepsilon^2}\frac{M_1\ldots M_r}{N^r}$ there holds
    \begin{equation}\label{prop4.4}
        \left|\gamma_\lambda(q)-\mathsf{Var}\left[e(\xi,\lambda)\right]\delta_0(q)\right|\leq \varepsilon.
    \end{equation}
\end{prop}
\begin{proof}[End of proof of \Cref{basicthm}:]
    Consider the small probability event $\mathcal A_\epsilon$ for which \eqref{prop4.4} holds. Conditioned on $\mathcal A_\epsilon^c$, we use the pointwise bound of Proposition \eqref{variance}, triangular inequality over $H$ and \eqref{eqgamma}, obtaining:
    \begin{equation}\label{Ghatbound}
        \left|N^{d-r}(M_1\ldots M_r)|\widehat{G_{\lambda}}(x)|^2 - \mathsf{Var}\left[e(\xi,\lambda)\right]\widehat{\delta_0}(x)\right|\le \varepsilon\frac{|H|}{N^d}=\frac{\varepsilon  }{N^{d-r}(M_1\ldots M_r)}.
    \end{equation}
    By substituting \eqref{Ghatbound} in \eqref{hatfhatgerror} we obtain the following estimation for $\langle f_{\lambda},g_{\lambda}\rangle_{l^2(H)}$ valid in $\Omega\setminus\mathcal A_\epsilon$:
    \begin{eqnarray}\label{chainineq}
        \frac{|\langle f_{\lambda},g_{\lambda}\rangle_{\ell^2(H)}|}{N^d}
        &\leq&\frac{1}{M_1\ldots M_r}\frac{1}{N^{d-r}}\sum_{y\in H^*-\lambda}|\widehat{G_{\lambda}(y)}|
        \nonumber\\
        &\leq&\frac{1}{N^{\frac{3}{2}(d-r)}(M_1\ldots M_r)^{\frac{3}{2}}}\sum_{y\in H^*-\lambda}\left(\mathsf{Var}[e(\xi,\lambda)]\widehat{\delta_0}(y)+\frac{\varepsilon}{N^{d-r}(M_1\ldots M_r)}\right)^{1/2}\nonumber\\
        &=&\frac{|H^*|}{N^{\frac{3}{2}(d-r)}(M_1\ldots M_r)^{\frac{3}{2}}}\left(\frac{\mathsf{Var}[e(\xi,\lambda)]}{N^d}+\frac{\varepsilon}{N^{d-r}(M_1\ldots M_r)}\right)^{1/2}\nonumber\\
        &=& \frac{1}{N^{\frac{d-r}{2}}(M_1\dots M_r)^{\frac{3}{2}}}\left(\frac{\mathsf{Var}[e(\xi,\lambda)]}{N^d}+\frac{\varepsilon}{N^{d-r}(M_1\ldots M_r)}\right)^{1/2}\nonumber\\
        &\leq & \frac{1}{N^{d-r}(M_1\dots M_r)}\left(\varepsilon+\frac{C M_1\ldots M_r}{N^r}\right)^{1/2}\nonumber\\
        &=& \frac{1}{|H^*|}\left(\varepsilon+\frac{C}{|H|}\right)^{1/2}\nonumber.
    \end{eqnarray}
    The claimed bound \eqref{mainerror2} follows directly, completing the proof of \Cref{basicthm}. 
\end{proof}
\subsection{Almost sure recovery of discrete subgroups in $\mathbb R^d$ as a limit case}\label{sec: eucrecov}
In this section, we see how the quantitative recovery given by \Cref{basicthm} implies almost sure recovery of the Fourier spectrum for discrete subgroups of $\mathbb{R}^d$. 
By a ($r$-dimensional) discrete subgroup, we mean a subgroup $X\leq\mathbb{R}^d$ of the form $H=v_1\mathbb{Z}\oplus\ldots\oplus v_r\mathbb{Z}$, where $\{v_1,\ldots,v_d\}$ is a basis of $\mathbb{R}^d$ and $r\leq d$.
\medskip 

By a (complex) measure, we mean a linear functional on the set of compactly supported $C_c(\R^d)$ such that for every compact set $K\subset\R^d$, there is a positive constant $M_K$ such that 
\[|\mu(f)|\leq M_k\|f\|_{\infty},\]

for every $f\in C_c(\R^d)$ supported on $K$, and where $\|\cdot\|_{\infty}$ denotes the supremum norm in $C_c(\R^d)$. By the Riesz Representation Theorem, there is an equivalence between this definition of measure and the classical measure-theoretic concept of regular Radon measure. 
\medskip

Observe that in this case, the co-adjoint map $\mathsf{Ad}^*$ reduces to the identity, and the space of coadjoint orbits $\mathfrak{g}/\mathsf{Ad}^*G$ is equals to $\R^d$. Then, the elements of $\widehat{\R^d}$ correspond to the $1$-dimensional irreducible representations 
\[
\chi_{\lambda}:=e^{-2\pi i\langle\cdot,\lambda\rangle},\qquad\text{where }\lambda\in\R^d,
\]
which act naturally over $L^2(\R)$; thus, in our case we have that $s:=\mathsf{dim}(G/K_{\lambda})=1$.
\medskip

For a finite measure $\mu$ over $\R^d$, its Fourier transform is defined by

\begin{equation*}
    \widehat{\mu}(\lambda):=\int_{\R^d}e^{-2\pi i\langle x,\lambda\rangle}d\mu(x),\qquad\text{for }\lambda\in\R^d.
\end{equation*}
\medskip

For $H$ being an $r$-dimensional discrete subgroup of $\mathbb{R}^d$, $r\leq d$, we define its {\em (averaged) Fourier transform} $\mathcal{F}(\delta_H)$ as follows
\begin{equation}\label{FT}
\begin{split}
    (\forall\lambda\in\R^d),\qquad\mathcal{F}(\delta_H)(\lambda)
    &:=\lim_{R\to\infty}\frac{\mathcal{F}(\delta_{H\cap [-R/2,R/2)^d})(\lambda)}{R^d}=\lim_{R\to\infty}\frac{1}{R^{r}}\sum_{p\in H\cap [-R/2,R/2)^d}e^{-2\pi i\langle p,\lambda\rangle}
\end{split}    
\end{equation}
For an $r$-dimensional discrete subgroup $H$, consider i.i.d random vectors $(\xi_p)_{p\in H}$ defined in a probability space $(\Omega, \mathcal{T},\mathbb P)$, and the corresponding realization set $ X_{\xi}=\{p+\xi_p:\ p\in H\}$. 
We define the Fourier Transform of $H_{\xi}$ in duality with $C_c(\mathbb{R}^d)$ like in \eqref{FT} by
\begin{equation}\label{FTxi}
    \mathcal{F}\left(\delta_{H_{\xi}}\right)
    :=\lim_{R\to\infty}\frac{1}{R^r}\mathcal{F}\left(\delta_{H_{\xi}\cap \left[-R/2,R/2\right)^d}\right).
\end{equation}
Thus, in this section our main result is the almost sure recovery of $H$ from its random perturbations under suitable moment conditions of the perturbations $\xi_p$.
\begin{thm}\label{thm:contsquare}
    Assume that $1\le r\le d$, that $H\subset \R^d$ is an $r$-dimensional discrete subgroup, and $\varepsilon>0$ is such that $\mathbb E[|\pi_r(\xi)|^{r+\varepsilon}]<\infty$, where $\pi_r:\R^d\to\R^r$ is the orthogonal projection onto the first $r$ coordinates.
    Then almost surely, there holds:
    \[
        \mathcal{F}\left(\delta_{H_{\xi}}\right)=\mathbb E\left(e^{-2\pi i\langle\xi,\cdot\rangle}\right)\mathcal{F}\left(\delta_{H}\right).
    \]
\end{thm}
We will restrict ourselves to the case $H=\mathbb Z^r\subset \mathbb{R}^d$, $r\leq d$; the case of more general discrete subgroups $H=v_1\mathbb{Z}\oplus\ldots \oplus v_r\mathbb{Z}$ then follows by a change of basis via a suitable $A\in GL(\mathbb{R}^d)$ that sends $\{e_1\ldots,e_r\}$ to $\{v_1,\ldots v_r\}$, and is left to the reader.
\medskip

For $M, N$ being positive integers, consider the torus
\[
    H_N=\mathbb T^r_N:=(\mathbb Z/N\mathbb Z)^r\subset\mathbb T_{M,N}^d:=\left(\left(\frac1M\mathbb Z\right)/N\mathbb Z\right)^d.
\]
\subsubsection{Rescaled DFT}\label{fmn}
Now we check how the above rescaling is inherited to the discrete Fourier transforms of interest to us.
Consider the torii $\mathbb T_{MN}^d:=(\Z/MN\Z)^d$ and $\widetilde{H}:=(M\Z/NM\Z)^r\simeq(M\Z/NM\Z)^r\times\{0\}^{d-r}\subseteq \mathbb T_{MN}^d$ first. Then $NM,M$ plays the role of $N,M$ of the previous section. 
In this case $\widetilde{H}^*=(N\mathbb Z/NM\mathbb Z)^r\oplus(\mathbb{Z}/MN\mathbb{Z})^{d-r}$, and we have $|\widetilde{H}|=N^r,|\widetilde{H}^*|=M^dN^{d-r}$. 
For the case $r=d$, the continuum non-compact version of the problem is that in which $\widetilde{H}$ is replaced by a lattice $\Lambda\subset\mathbb R^d$, and $\widetilde{H^*}$ is replaced by its dual lattice $\Lambda^*$, whose density satisfies $\mathrm{dens}(\Lambda)\mathrm{dens}(\Lambda^*)=1$.
\medskip 

In the setting of \Cref{ssection: setting} (with the roles of $N, M_1,\dots,M_r$ replaced here by $MN, M,\dots,M$), the DFT of $f:\mathbb T_{MN}^d\to\mathbb C$ is given by \eqref{dftdef} as
\begin{equation}\label{DFT3}
    \mathcal F_{MN}(f)(k):=\frac{1}{M^dN^d}\sum_{p\in\mathbb T_{MN}^d}e^{\frac{-2\pi i}{MN}p\cdot k}f(p)\qquad \text{for }k\in\mathbb T_{MN}^d,
\end{equation}
For algebraic considerations (i.e., without including metric aspects), definition \eqref{DFT3} would suffice. 
However, our study is based on the fact that, as $M\to\infty$, the metric spaces $\mathbb T^d_{M,N}$ approximate better and better $\R^d/N\Z^d$ (since $\mathbb T^d_{M,N}$ is a $1/M$-net of $\R^d/N\Z^d$, these spaces converge to each other in Gromov-Hausdorff sense, for example). 
Thus we are led to define
\begin{equation}\label{emn}
    e_{M,N}(p,k):=e^{-\frac{2\pi i}{N}\langle p, k\rangle}, \quad p\in\mathbb T^d_{M,N}.
\end{equation}

At the level of Fourier-dual spaces, we observe that the compact group $\R^d/N\Z^d$ is naturally in duality with $\frac1N \Z^d$: 
functions over $\R^d/N\Z^d$ correspond to $N\Z^d$-periodic functions over $\R^d$, and the smallest frequency that can be sampled via these functions scales like $1/N$. 
In other words, the Pontryagin dual of $\mathbb{T}_{M,N}^d$ can be identified with $\mathbb{T}_{N,M}^d$, where the family of characters is given by $e_{M,N}(\cdot,Nk'),\ k'\in\mathbb{T}_{N,M}^d$. 
We then use the definition $\mathcal F_{M,N}(f)=\mathcal F_{MN}\left(f\circ(\mathsf{dil}_{1/M})^{-1}\right) \circ \mathsf{dil}_{1/N}$, where for $a>0$ the map $\mathsf{dil}_a:\mathbb T^d_{MN}\to a\mathbb T^d_{MN}$ is defined by $\mathsf{dil}_a(p)=ap$.
Explicitly, we have then
\begin{equation}\label{DFT4}
    \mathcal F_{M,N}:\ell^2(\mathbb T^d_{M,N})\to\ell^2(\mathbb T^d_{N,M}),\quad \mathcal F_{M,N}(f)(k'):=\frac1{M^dN^d}\sum_{p\in\mathbb T^d_{M,N}}e_{M,N}(p,N k') f(p).
\end{equation}
If $r=d$, one can check that $(\mathcal F_{M,N})^{-1}(g)=\mathcal F_{N,M}(\widetilde{g})$ where $\widetilde g(k):=\overline{g(-k)}$. 
\medskip

As we showed before, \Cref{thm:contsquare} is a consequence of the fact that the Fourier transforms of the Dirac measures supported on $H=\mathbb Z^r$ and $H_{\xi}\mathbb Z^r_{\xi}$ correspond to the (normalized) limit of the DFT of $H_N$ and $H_{\xi^{(M,N)}}$. \Cref{DFT-FT} below is the adapted version of \Cref{nilDFT-FT} in the Euclidean setting, where $s=\mathsf{dim}(G/K_l)=1$, with the definitions of \Cref{sec: lattnil}.

\begin{lemma}\label{DFT-FT}
    Let $(\xi_p)_{p\in\Z^r}$ be i.i.d random vectors and let $\varepsilon>0$ be such that $\mathbb E[|\pi_r(\xi)|^{r+\varepsilon}]<\infty$, where $\pi_r:\mathbb R^d\to\mathbb R^r$ is the orthogonal projection onto the first $r$ coordinates.
    Let $H_N := (\Z/N\Z)^r\subset \mathbb T^d_{M,N}$ and let $\xi^{(M,N)}$ be as in Subsection \ref{ximnnil}. 
    Then almost surely, there holds:
    \begin{equation}\label{lemmalimDFT}
    \begin{split}
        \lim_{N\to\infty}\lim_{M\to\infty}M^{d-1}N^{d-r}\mathcal F_{M,N}(\delta_{H_N})=\mathcal{F}(\delta_{\mathbb Z^r})\ \text{ and }\ \lim_{N\to\infty}\lim_{M\to\infty}M^{d-1}N^{d-r}\mathcal F_{M,N}\left(\delta_{H_{\xi}}\right)=\mathcal{F}(\delta_{\mathbb Z^r_{\xi}}),
    \end{split}
    \end{equation}
where the limit for $M\to\infty$ is taken in the sense of \Cref{weaknilconv}, and the limit with respect to 
$N\to\infty$ is in the classical sense.  
\end{lemma}
To prove the recovery result of \Cref{thm:contsquare}, we re-express \Cref{basicthm} and \Cref{corobasic} for the torii $\mathbb{T}_{M,N}^d$ and $X=\mathbb{T}_N^r$. To do this we replace $(N,M_1,\dots,M_r)\mapsto (NM, M,\dots,M)$ obtaining the following result:
\begin{prop}\label{basicthm2}
    Let $\xi_p, p\in H_N=\mathbb{T}_N^r$ be i.i.d. random variables with $\xi_p\sim\xi$ satisfying $C:=\max_{\lambda\in\mathbb T^d_{N,M}}\mathsf{Var}[e(\xi,\lambda)]<\infty$, and assume that $N^r>12$ and $\varepsilon>\sqrt{\frac{72}{N^r}}$.
    Then with probability at least $1-\frac{72}{\varepsilon^2 N^r}$ we have 
    \begin{equation}\label{mainerror2'}
        \max_{\lambda\in \mathbb T_{N,M}^d}\left|\mathcal{F}_{M,N}(\delta_{(H_N)_\xi})(\lambda) - \mathbb E\left[e(\xi,\lambda)\right]\ \mathcal{F}_{M,N}(\delta_{H_N})(\lambda)\right|\le \frac{1}{M^dN^{d-r}}\sqrt{\varepsilon+\frac{4C^2}{N^r}}.
    \end{equation}
\end{prop}
\begin{proof}[Proof of \Cref{thm:contsquare}:]
    Following \Cref{DFT-FT}, we need an estimation for the error in \Cref{basicthm2}, after multiplying the left-hand side of \eqref{mainerror2} by $M^{d-1}N^{d-r}$. Let $\xi^{(M,N)}$ be as in the Subsection \ref{ximnnil} and $\varepsilon>\sqrt{\frac{72}{N^r}}$ with $N^r>12$. From \eqref{mainerror2}, with probability at least $1-\frac{72}{\varepsilon^2 N^r}$ we get:
    \begin{equation}\label{mainerror3}
        \max_{\lambda\in \mathbb T_{N,M}^d}\left|M^{d-1}N^{d-r}\left(\mathcal{F}_{M,N}(\delta_{H_{\xi^{(M,N)}}})(\lambda) - \mathbb E\left[e(\xi^{(M,N)},\lambda)\right]\ \mathcal{F}_{M,N}(\delta_{H_N})(\lambda)\right)\right|\leq\frac{1}{M}\sqrt{\varepsilon+\frac{4C^2}{N^r}},
    \end{equation}
    where we use the fact that $\max_{\lambda\in\mathbb{T}_{N,M}^d}\mathsf{Var}[e(\xi^{(M,N)},\lambda)]\leq 4$. 
    We now bound \eqref{mainerror3} in the regime where $\varepsilon\to 0$ and $M,N\to\infty$. We require $N^{-r/2}\ll\varepsilon$, so that the event \eqref{mainerror3} has probability tending to $1$. Hence, by letting $M\to\infty$ and $\varepsilon\to 0$ such that $\varepsilon N^r\gg 1$ we get that almost surely
    \begin{equation}\label{asbasic3}
        \left\|M^{d-1}N^{d-r}\left(\mathcal{F}_{M,N}(\delta_{H_{\xi^{(M,N)}}})-\mathbb E[e(\xi^{(M,N)},\cdot)]\mathcal{F}_{M,N}(\delta_{H_N})\right)\right\|_{\ell_\infty(\mathbb T_{N,M}^d)}\to 0.
    \end{equation}
    \Cref{thm:contsquare} follows from it and \Cref{DFT-FT} via triangle inequality and from \Cref{limitxi}.
\end{proof}
%
%%%%%%%%%%%%%%%%%%%%%%%%%%%%%%%%%%%%%%%%%%%%%%%%%%%%%%%%%%%%%%%%%%%%%%%%%%%%%%%%%%%%%%%%%%%%%%%%%%%%%%%%%%%%%%%%%%%%%%%%%%%%%%%%%%%%%%%%%%%%%%%%
\subsection{The Heisenberg group case}\label{sec:heis}
In this section, for our purposes we consider the (polarized) Heisenberg group $\Gamma$ being the set $\mathbb{R}^3$ with the group law
\begin{equation}\label{opheis}
    (x,y,z)\cdot (u,v,w):=\left(x+u,y+v,z+w+xv\right).
\end{equation}
\medskip
Similarly, for a number field $F$, we define $\Gamma(F):=\{(x,y,z)\in\Gamma:\ x,y,z\in F\}$, with group operation as in \eqref{opheis}. We consider especially $H_N=\Gamma(\mathbb{Z}/N\mathbb{Z})$ as the finite Heisenberg group with coefficients in the field $\mathbb{Z}/N \mathbb{Z}$, with sum and product operations as usual; observe that $H_N\cong\Gamma(\mathbb{Z})/\Gamma(N\mathbb{Z})$. Note that the Lebesgue measure in $\mathbb{R}^3$ is the Haar measure in $\Gamma$. Throughout this section, for $R>0$ we denote 
\[
    B_R:=[-R/2,R/2)^3.
\]
Consider the usual Heisenberg group $\Gamma'=(\mathbb{R}^3,\odot )$ with the group operation given by
\begin{equation}
    (x,y,z)\odot (u,v,w):=\left(x+u,y+v,z+w+\frac{1}{2}(xv-yu)\right).
\end{equation}
The Fourier Analysis in $\Gamma'$ is well-understood. Indeed, denote by $\mathcal{U}(L^2(\mathbb{R}))$ be the group of unitary operators acting on $L^2(\mathbb{R})$. By the Stone-von Neumann Theorem (see, for instance \cite[Theorem 6.49]{folland}), all the irreducible unitary representations of $\Gamma'$ (under equivalence) are determined by one of the following:
\begin{enumerate}
    \item The representations $\pi'_{\lambda}:\Gamma'\to\mathcal{U}(L^2(\mathbb{R}))$, $\lambda\neq 0$, where for $(x,y,z)\in\Gamma'$, the operator $\pi'_{\lambda}(x,y,z)$ acts on a function $\varphi\in L^2(\mathbb{R})$ via norm-preserving operations, explicitly given by the formula
\[
    \pi'_{\lambda}(x,y,z)\varphi(u):=e^{2\pi i\lambda z+\pi i\lambda xy}e^{2\pi i\lambda yu}\varphi(u+x),\qquad u\in\mathbb{R}.
\]
\item The 1-dimensional irreducible unitary representations $\pi'_{\alpha,\beta}:\Gamma'\to\mathbb{C}$, $\alpha,\beta\in\mathbb{R}$, given by

\[
\pi'_{\alpha,\beta}(x,y,z)=e^{2\pi i(\alpha x+\beta y)}.
\]
\end{enumerate}
The connection between $\Gamma'$ and $\Gamma$ is given by the isomorphism $\phi:\Gamma'\to\Gamma$:
\[
    \phi(x,y,z):=\left(x,y,z+\frac{1}{2}xy\right).
\]
This isomorphism, together with the classification of elements in $\widehat{\Gamma'}$, induces a bijection at the level of irreducible unitary representations which maps $\pi'_{\lambda}, \pi'_{\alpha,\beta}\in\widehat{\Gamma'}$ to the representations $\pi_{\lambda},\pi_{\alpha,\beta}\in\widehat{\Gamma}$ given by
\[
    \pi_{\lambda}(x,y,z):=\pi'_{\lambda}\left(x,y,z-\frac{1}{2}xy\right)\qquad\text{and}\qquad\pi_{\alpha,\beta}(x,y,z):=\pi'_{\alpha,\beta}\left(x,y,z-\frac{1}{2}xy\right),
\]
respectively. More precisely, the irreducible representations $\pi_{\lambda}(x,y,z)$ and $\pi_{\alpha,\beta}(x,y,z)$ act over a function $\varphi\in L^2(\mathbb{R})$ by:
\begin{equation}
    \pi_{\lambda}(x,y,z)\varphi(u)=e^{2\pi i\lambda z}e^{2\pi i\lambda yu}\varphi(u+x)\qquad\text{and}\qquad \pi_{\alpha,\beta}(x,y,z)=e^{2\pi i(\alpha x+\beta y)}.
\end{equation}
Following \eqref{FTnil}, for the lattice $H=(\Z^3,\cdot)\subset\Gamma$ we define its (averaged) Fourier Transform $\mathcal{F}(\delta_H)$ by the following limit in the weak-sense:
\begin{equation}\label{wFTH}
    \mathcal{F}(\delta_H)(\lambda):=\lim_{R\to\infty}\frac{\mathcal{F}(\delta_{H\cap B_R})(\pi_{\lambda})}{R^3}=\lim_{R\to\infty}\frac{1}{R^3}\sum_{(x,y,z)\in H\cap B_R}\pi_{\lambda}(x,y,z)
\end{equation}
and
\begin{equation}\label{wFTH2}
    \mathcal{F}(\delta_H)(\alpha,\beta):=\lim_{R\to\infty}\frac{\mathcal{F}(\delta_{H\cap B_R})(\pi_{\alpha,\beta})}{R^3}=\lim_{R\to\infty}\frac{1}{R^3}\sum_{(x,y,z)\in H\cap B_R}\pi_{\alpha,\beta}(x,y,z).
\end{equation}
More precisely, the limit \eqref{wFTH} is defined in duality with functions with compact support, namely, for every $\varphi,\psi\in C_c(\mathbb{R})$ there holds
\[
    \langle\mathcal{F}(\delta_H)(\lambda)\varphi,\psi\rangle_{L^2(\mathbb{R})}=\lim_{R\to\infty}\frac{1}{R^3}\sum_{(x,y,z)\in H\cap B_R}\int_{\mathbb{R}}e^{2\pi i\lambda(yu+z)}\varphi(x+u)\overline{\psi(u)}du;
\]
whereas for $\alpha,\beta\in\mathbb{R}$ we have that
\[
    \mathcal{F}(\delta_H)(\alpha,\beta)=\lim_{R\to\infty}\frac{1}{R^3}\sum_{(x,y,z)\in H\cap B_R}e^{2\pi i(\alpha x+\beta y)}.
\]
For the lattice $H=(\mathbb{Z}^3, \cdot)$,
consider the i.i.d random vectors $(\xi_{p})_{p\in H}$ defined in a probability space $(\Omega,\mathcal{T},\mathbb{P})$ with common distribution $\xi$, and the randomly perturbed set
\[
    H_{\xi}=\{p\cdot\xi_p:\ p\in H\}\subset\Gamma. 
\]
With these notations, our main recovery result in the Heisenberg group reads as follows.
\begin{thm}\label{Heisrecov}
Under the previous notations, assume that there is a positive number $\varepsilon$ such that $\mathbb{E}(|\xi|^{3+\varepsilon})<\infty$. Then almost surely, for every non-zero $\lambda\in\mathbb{R}$ and $\alpha,\beta\in\mathbb{R}$, there holds:
\begin{equation*}\label{recoveqheis}
   \mathcal{F}(\delta_{H_{\xi}})(\lambda)=\mathcal{F}(\delta_H)(\lambda)\mathbb{E}[\pi_{\lambda}(\xi)]\qquad\text{and}\qquad \mathcal{F}(\delta_{H_{\xi}})(\alpha,\beta)=\mathcal{F}(\delta_H)(\alpha,\beta)\mathbb{E}[\pi_{\alpha,\beta}(\xi)].
\end{equation*}
\end{thm}
Like in the Euclidean case, \Cref{Heisrecov} will be a consequence of the approximate recovery given by \Cref{basicthmH} for the finite Heisenberg groups
\[
    \Gamma_{M,N}:=\Gamma\left(\left(\frac{1}{M}\mathbb{Z}\right)/N\mathbb{Z}\right)\quad\text{and}\quad H_N:=\Gamma(\mathbb{Z}/N\mathbb{Z}).
\]
Observe that $|\Gamma_{M,N}|=M^3N^3$ and $|H_N|=N^3$.
\medskip

\noindent{\bf Irreducible representations in $H_n$ and scaling reasoning: } In this case we take $m|n$ and consider
\[
    x,y,z\in\mathbb Z/n\mathbb Z,\quad a,b\in\mathbb Z/\frac{n}{m}\mathbb Z,\quad c\in\mathbb Z/m\mathbb Z\text{ with }\mathsf{gcd}(c,m)=1,\quad f:\mathbb Z/m\mathbb Z\to \mathbb C,
\]
and denoting $e(x):=e^{2\pi i x}$ we have representation formula
\[
    \pi_{a,b,c}(x,y,z)f(j)= e\left(\frac{ax+by}{n}\right)e\left(\frac{c(yj + z)}{m}\right) f(j+x).
\]
Observe that by the definition of $\mathsf{gcd}$, the only case when $c=0$ corresponds to $m=1$; in this case, $\pi_{a,b,0}$ yields the $1$-dimensional representation:
\[
    \pi_{a,b}(x,y,z):=\pi_{a,b,0}(x,y,z)=e\left(\frac{ax+by}{n}\right).
\]
As we pass to the continuum we rescale $H_n\mapsto \Gamma_{M,N}$ with $n=MN$, and we rescale the representations $\pi_{a,b,c}$ in order to recover continuum representations $\pi_\lambda, \pi_{\alpha,\beta}$ as in \eqref{wFTH}, \eqref{wFTH2}.

\medskip
\noindent{\bf Irreducible unitary representations in $\Gamma_{M,N}$: } To obtain the quantitative recovery for the finite Heisenberg group $\Gamma_{M,N}$, we need a description of its unitary irreducible representations. 
\medskip

Fix $m|M$, $c\in\mathbb{Z}/mM\mathbb{Z}$ with $\mathsf{gcd}(c,Mm)=1$, and $a,b\in\frac{1}{N} \mathbb{Z}/\frac{M}{m}\mathbb{Z}$. Define the unitary irreducible representation $\pi_{a,b,c,m}:\Gamma_{M,N}\to\mathcal{U}\left(\ell^2\left(\frac{1}{M}\mathbb{Z}/m\mathbb{Z}\right)\right)$ by
\begin{equation}\label{heisunitrep}
    \pi_{a,b,c,m}(x,y,z)f(j):=e\left(ax+by\right)e\left(\frac{c(yj+z)}{m}\right)f(j+x),\qquad\text{where }f\in \ell^2\left(\frac{1}{M}\mathbb{Z}/m\mathbb{Z}\right).
\end{equation}

Following \Cref{ssection: nonabelFT}, the (normalized) Discrete Fourier Transform of a function $f\in \ell^2\left(\Gamma_{M,N}\right)$ associates to every $a,b\in\frac{1}{N}\mathbb{Z}/\frac{M}{m}\mathbb{Z}$ and $c\in\mathbb{Z}/mM\mathbb{Z}$ the operator over $\ell^2\left(\frac{1}{M}\mathbb{Z}/m\mathbb{Z}\right)$  given by:
\begin{equation}\label{DFTHeis}
\mathcal{F}_{M,N}(f)(a,b,c):=\mathcal{F}_{M,N}(f)(\pi_{a,b,c,m})=\frac{1}{M^3N^3}\sum_{(x,y,z)\in\Gamma_{M,N}}f(x,y,z)\pi_{a,b,c,m}(x,y,z);
\end{equation}

We point out that formula \eqref{DFTHeis} only works in duality with continuous functions supported on $\ell^2\left(\frac{1}{M}\mathbb{Z}/m\mathbb{Z}\right)$. Since the Fourier Transform $\mathcal{F}(\delta_H)$ defined as in \eqref{wFTH} and \eqref{wFTH2} acts in duality with compactly supported functions, to establish a limit from $\mathcal{F}_{M,N}(\delta_{H_N})$ to $\mathcal{F}(\delta_H)$ we require the discrete-to-continuum weak-limit in \Cref{weaknilconv}.

\begin{lemma}\label{DFT-FT-Heis}
    Let $\lambda$ be a non-zero real number and $k\in\mathbb{N}$ be fixed. For every positive integer $m$, let $\Gamma_{km,N}=\Gamma\left(\frac{1}{km}\mathbb{Z}/N\mathbb{Z}\right)$, and let $H=(\Z^3,\cdot)$ and $H_N:= \Gamma(\Z/N\Z)\leq \Gamma_{km,N}$. Let $a,b\in\frac{1}{N}\mathbb{Z}/k\mathbb{Z}$ and $c_m\in\mathbb{Z}/km^2\mathbb{Z}$ be a sequence such that $\mathsf{gcd}(c_m,km^2)=1$ and $c_m/m\to\lambda$ when $m\to\infty$, then the following limit holds
    \begin{equation}\label{lemmalimDFT-Heis}
    \begin{split}
        \lim_{m\to\infty}(km)^2\mathcal{F}_{km,N}(\delta_{H_N})(\pi_{a,b,c_m,m})=\frac{1}{N^3}\sum_{(x,y,z)\in H\cap B_N}\pi_{a,b}(x,y,z)\pi_{\lambda}(x,y,z).
    \end{split}
    \end{equation}
    In particular, fix $k\in\mathbb{N}$, $\alpha,\beta\in[-k/2,k/2]$, and let $a_N,b_N\in\frac{1}{N}\mathbb{Z}/k\mathbb{Z}$ such that $a_N\to\alpha$ and $b_N\to\beta$ when $N\to\infty$. Then the following limits hold %in $L^2\left(\left[-\frac{m_0}{2},\frac{m_0}{2}\right]\right)$:
    \begin{equation}\label{lemmalimDFT-Heis2}
    \begin{split}
        \lim_{N\to\infty}\lim_{m\to\infty}(km)^2\mathcal{F}_{km,N}(\delta_{H_N})(\pi_{0,0,c_m,m})&=\mathcal{F}(\delta_H)\left(\lambda\right)\\
        \lim_{N\to\infty}\lim_{m\to\infty}(km)^2\mathcal{F}_{km,N}(\delta_{H_N})(\pi_{a_N,b_N,0,m})&=\mathcal{F}(\delta_H)\left(\alpha,\beta\right),
    \end{split}
    \end{equation}
    where all the limits for $m\to\infty$ are in the sense of \Cref{weaknilconv} and the limit with respect to $N$ is weak in the classical sense.
\end{lemma}

\begin{lemma}\label{lemmalimxiDFT-Heis}
    Let $\lambda\in\mathbb{R}\setminus\{0\}$, $\alpha,\beta\in\mathbb{R}$ and $m_0\in\mathbb{N}$ large enough such that $\lambda\in[-m_0/2,m_0/2]$. 
    Consider $a_N,b_N\in\frac{1}{N}\mathbb{Z}/k\mathbb{Z}$ and $c_m$ be as in \Cref{DFT-FT-Heis}. Let $\xi^{(km,N)}$ be as in \Cref{ximnnil} and suppose there exists a positive number $\varepsilon$ such that $\mathbb{E}[|\xi|^{3+\varepsilon}]<\infty$.  Then the following weak limits hold almost surely :
    \begin{equation}
    \begin{split}
        \lim_{N\to\infty}\lim_{m\to\infty}(km)^2\mathcal F_{km,N}\left(\delta_{H_{\xi^{(km,N)}}}\right)(\pi_{0,0,c_m,m})&=\mathcal{F}(\delta_{H_{\xi}})(\lambda)\\
        \lim_{N\to\infty}\lim_{m\to\infty}(km)^2\mathcal F_{km,N}\left(\delta_{H_{\xi^{(km,N)}}}\right)(\pi_{a_N,b_N,0,m})&=\mathcal{F}(\delta_{H_{\xi}})(\alpha,\beta).
    \end{split}
    \end{equation}
\end{lemma}
The last ingredient to prove the main recovery result on the Heisenberg group is that the error given by \Cref{basicthmH} for finite Heisenberg groups converges to zero when considering a sequence of discrete Heisenberg-torii $\Gamma_{M,N}$, for $M,N\to\infty$.
\begin{prop}\label{corobasic-heis}
    Let $N,m,k\in\mathbb{N}$ and consider finite Heisenberg groups $\Gamma_{km,N}$ and $H_{N}$ as before. Let $\xi_p^{(N)}=\xi_p,p\in H_N$ and $\varepsilon>0$ satisfying the hypotheses of \Cref{basicthmH}. Then with probability at least $1-\frac{C}{\varepsilon^2|H|}$ for some suitable positive constant $C$ there holds:
    \begin{equation}\label{asbasic-heis}
        \left\|\mathcal{F}_{km,N}(\delta_{H_{\xi}})-\mathcal{F}_{km,N}(\delta_{H})\mathbb{E}(\xi(\cdot))\right\|^2_{L^2(\widehat{\Gamma}_{km,N})}\leq\frac{(\varepsilon+4C^2)N^{7}}{(km)^{5}}.  
    \end{equation}
\end{prop}
\begin{proof}
    From \Cref{basicthmH} and since for the finite Heisenberg groups $\Gamma_{km,N}$ there holds that $\max\{d_{\pi}:\ \pi\in\widehat{\Gamma}_{km,N}\}=kmN=|\Gamma_{km,N}|^{1/3}$, then for $\frac{C'}{|H|}<\varepsilon^2$ we get:
    \begin{equation}\label{asest}
    \begin{split}
        \left\|\mathcal{F}_{km,N}(\delta_{H_{\xi}})-\mathcal{F}_{km,N}(\delta_H)\mathbb{E}(\xi(\cdot))\right\|^2_{L^2(\widehat{\Gamma}_{km,N})}&\leq\frac{|H|^4}{|\Gamma_{km,N}|^2}\left(\varepsilon+\frac{4C^2}{|H|}\right)\max_{\pi\in\widehat{\Gamma}_{km,N}}d_{\pi}\\
        &=\frac{\varepsilon N^{7}}{(km)^5}+\frac{4C^2N^{4}}{(km)^{5}}.
    \end{split}
    \end{equation}  
    To conclude, it suffices to use the fact that $\frac{N^{4}}{(km)^{5}}\leq \frac{N^{7}}{(km)^5}$.
\end{proof}
Hence, from \eqref{dualnorm-l2}, \eqref{asest2} and \Cref{corobasic-heis} we get the next result which in fact implies \eqref{FTasconv}.
\begin{corollary}\label{corobasic-heis2}
    Under the hypotheses and notations of \Cref{corobasic-heis}, let $\xi_p^{(km,N)}, p\in H_N$ as in \eqref{ximnnil}.
    Then, with probability at least $1-\frac{C}{\varepsilon^2|H|}$, for every $\lambda\in\mathbb{R}\setminus\{0\}$, $a,b\in \frac{1}{N}\mathbb{Z}/k\mathbb{Z}$ and every sequence $c_m\in\mathbb{Z}/km^2\mathbb{Z}$, $\mathsf{gcd}(c_m,km^2)=1$ for which $c_m/m\to\lambda$ when $m\to\infty$, there holds that:
    \begin{equation*}
     (km)^2\left\|\mathcal{F}_{km,N}(\delta_{H_{\xi^{(km,N)}}})(\pi)\varphi-\mathcal{F}_{km,N}(\delta_{H_N})\mathbb{E}(\pi(\xi^{(km,N)}))\varphi\right\|_{\ell^2\left(\frac{1}{km}\mathbb{Z}/m\mathbb{Z}\right)}\longrightarrow 0,
     \end{equation*}
     where $\pi$ is of the form $\pi_{0,0,c_m,m}$ or $\pi_{a,b,0,m}$ and $\phi\in\ell^2\left(\frac{1}{km}\Z/m\Z\right)$.
\end{corollary}
\begin{proof}[Proof of \Cref{Heisrecov}] 
It is a direct application of \Cref{DFT-FT-Heis}, \Cref{lemmalimxiDFT-Heis} and \Cref{corobasic-heis2}.
    
\end{proof}

{\bf Acknowledgements: }\emph{ M. P. was sponsored by the Chilean Fondecyt Regular grant number 1210462 entitled ``Rigidity, stability and uniformity for large point configurations'', and R. V. was sponsored by the Chilean Fondecyt Postdoctoral grant number 3210109 entitled ``Geometric and Analytical aspects of Discrete Structures''.}

%%%%%%%%%%%%%%%%%%%%%%%%%%%%%%%%%%%%%%%%%%%%%%%%%%%%%%%%%%%%%%%%%%%%%%%%%%%%%%%%%%%%%%%%%%%%%%%%%%%%%%%%%%%%%%%%%%%%%%%%%%%%%%%%%%%%%%%%

%%%%%%%%%%%%%%%%%%%%%%%%%%%%%%%%%%%%%%%%%%%%%%%%%%%%%%%%%%%%%%%%%%%%%%%%%%%%%%%%%%%%%%%%%%%%%%%%%%%%%%%%%%%%%%%%%%%%%%%%%%%%%%%%%%%%%%%%%%%%%%%%%%%%%%%%%%%%%%%%%%%%%%%%%%%%%%%%%%%%%%%%%%%%%%%%%%%%%%

\end{document}